\documentclass[12pt,a4paper]{article}

\usepackage{amssymb}
\usepackage{amsfonts}
\usepackage{graphicx}
\newcommand{\bbb}[1]{\mathbb #1}

\newcounter{Bildchen}
\newenvironment{remark*}{\medskip\noindent{\bf Remark.}}{\medskip}
\newenvironment{remarks*}{\medskip\noindent{\bf Remarks.}}{\medskip}

\newenvironment{proof}{\noindent{\sc
Proof: }}{\hfill{\scriptsize $\square$}\medskip}
\newcommand{\cp}{\circ\ppi}
\newcommand{\auf}[1]{\circ #1}

\newcommand{\mapon}{cell decomposition of }

\newcommand{\refpunkt}{.}
\newcommand{\refthpunkt}{.}

\newcommand{\cell}{cell}
\newcommand{\cells}{cells}

\newcommand{\Figur}[4]{
 \begin{minipage}{0.95\textwidth}
  \refstepcounter{Bildchen}
  \label{#1}
    \begin{center}
    \hspace*{#3em}
\includegraphics[height=#4cm]{#1.eps}
    \end{center}
    \begin{center}
    {\sc #2}
    \end{center}
    \medskip
 \end{minipage}}
\newenvironment{versetzt}{\begin{list}{}{\setlength{\leftmargin}{8mm}
\setlength{\topsep}{0mm}\setlength{\itemsep}{0mm}
\setlength{\parsep}{0mm}}}{\end{list}}

\newenvironment{vversetzt}[1]{\begin{list}{}{\setlength{\leftmargin}{#1 mm}
\setlength{\topsep}{0mm}\setlength{\itemsep}{0mm}
\setlength{\parsep}{0mm}}}{\end{list}}

\newcommand{\wt}{\widetilde}
\newcommand{\wh}{\widehat}
\newcommand{\wb}{\overline}

\newcommand{\NS}{{\bf S}^2}
\newcommand{\NH}{{\bf H}^2}
\newcommand{\NE}{{\bf E}^2}

\newcommand{\HY}[1]{{\bf H}^#1}
\def\ppi{\mathop{u\hspace*{-0.5ex}c}\nolimits}

\def\dist{\mathop{\rm dist}\nolimits}
\def\vol{\mathop{\rm vol}\nolimits}
\def\area#1{{\rm area}(#1)}
\def\mod{\mathop{\rm mod}\nolimits}

\def\tr#1{\mathop{\Delta_{#1}}\nolimits}

\newcommand{\new}[1]{{\em #1}}

\newcommand{\zweierindex}[2]{\scriptstyle #1 \atop\scriptstyle #2}

\def\stH#1{\mathop{({\rm{Iso}}\,{\HY 3})_{#1}}\nolimits}
\def\ch#1{\mathop{{|#1|}}\nolimits}

\newcommand{\CS}{{\bf S}^2}
\newcommand{\CO}{{\rm{conf}}\,{\bf S}^2}
\def\st#1{\mathop{({\rm{conf}}\,{\bf S}^2)_{#1}}\nolimits}
\newcommand{\MR}{{{\bf M}_{\rm reg}}}
\newcommand{\MS}{{{\bf M}_{\rm sing}}}
\newcommand{\MD}{{{\bf M}_{\rm disk}}}
\def\reg#1{\mathop{\rm reg}\nolimits(#1)}

\def\H#1{{\cal H}(#1)}

\def\Mi#1#2{{{\rm C}_{#1}( {#2})}}

\def\cK{{\cal A}}
\def\ccK{\cK(\wtV)}
\def\Dual{\cK}
\def\dual#1{\Dual(#1)}

\def\CH#1{{\vert #1\hspace{0.1ex}\vert}}

\newcommand{\V}{V(T)}
\newcommand{\E}{E(T)}
\newcommand{\F}{F(T)}
\newcommand{\Se}{S(T)}
\newcommand{\wtV}{V(\wt T)}
\newcommand{\wtE}{E(\wt T)}
\newcommand{\wtF}{F(\wt T)}
\newcommand{\wtS}{S(\wt T)}
\def\wttheta{{\wt\theta}}

\newcommand{\br}[1]{\langle #1 \rangle}
\newcommand{\proj}[1]{| #1 |}
\newcommand{\rev}[1]{-#1}

\newcommand{\alp}{\widehat\alpha}
\newcommand{\bet}{\widehat\beta}
\newcommand{\gam}{\widehat\gamma}

\newcommand{\WC}{{\cal F}_{\Delta,\Sigma}}
\newcommand{\WT}{{\cal F}_{\Delta}}
\newcommand{\WCf}{{\cal F}_{\Delta,\Sigma,f}}

\newcommand{\clWC}{\overline{{\cal F}}_{\Delta,\Sigma}}
\newcommand{\clWT}{\overline{{\cal F}}_{\Delta}}
\newcommand{\clWCf}{\overline{{\cal F}}_{\Delta,\Sigma,f}}

\def\ex{\mathop{\widehat\eta}\nolimits}

\def\LaT{\mathop{{\rm L}_{\Delta}}\nolimits}
\def\LaC{\mathop{{\rm L}_{\Delta,\Sigma}}\nolimits}
\def\LaCf{\mathop{{\rm L}_{\Delta,\Sigma,f}}\nolimits}
\def\lob{\mathop{\cal{L}}\nolimits}

\def\Imi#1{\mathop{{\cal I}}_{\mbox{\tiny$\pi\!\!-\!\!#1$}}\nolimits}
\def\Ino#1{\mathop{{\cal I}}_{\mbox{\tiny$#1$}}\nolimits}

\def\len{\mathop{l}\nolimits}
\newcommand{\sect}[1]{\mbox{{\footnotesize$(#1)$}}}

\newcommand{\ra}[0]{\rightarrow}
\newcommand{\ek}[1]{[\mbox{{\footnotesize$#1$}}]}

\def\Lim#1#2{{\bot}^{\!\!{#1}}_{\!\partial}#2}
\def\NorIndex#1#2#3{\langle\!\langle #1,#2,#3 \rangle\!\rangle}
\def\Index#1#2#3{\langle\!\langle\mbox{{\footnotesize$#1,#2,#3$}}\rangle\!\rangle}
\def\Cone#1#2{{\rm Cone}_{#1}\left(#2\right)}
\def\fac#1{{{\frak f}}({#1})}

\def\Fig#1#2{{\Omega_{#1}( {#2})}}
\def\ep#1#2{{\varepsilon_{#1}( {#2})}}
\def\Ort#1#2{{O_{#1}( {#2})}}
\def\Fau#1{{\cal V}(#1)}

\def\Dest{\Delta^*}

\def\wab{\omega_\gamma (\alpha,\beta)}
\def\wba{\omega_\gamma (\beta,\alpha)}
\def\wabst{\omega_{\gamma^*} (\alpha^*,\beta^*)}

\newlength{\figurwidth}
\begin{document}
\newcommand{\Paragraph}[1]{\section{#1}}
\newcommand{\abschnitt}[1]{\subsection{#1}}
\newenvironment{beweis}{\begin{proof}}{\end{proof}}
\newtheorem{theorem}{Theorem}%[section]
\newtheorem{theorem*}[theorem]{Theorem}
\newtheorem{lemma}{Lemma}
\newtheorem{proposition}{Proposition}
\newtheorem{proposition*}[proposition]{Proposition}

\author{Walter Br\"agger}
\title{A Uniformisation of Weighted Maps\\ on Compact Surfaces}

\maketitle

\bigskip

This article is a comprised version of my doctoral thesis completed under the supervision of Norbert A'Campo at Basel University in 1995. For various reasons, I never published the thesis up to now. After all with a delay of two-and-a-half decades, I decided to catch up on it.

\vskip3em

\begin{abstract}
In 1970 Andreev [An] proved a theorem concerning the existence 
of polyhedra with preassigned 
dihedral angles in the hyperbolic 3-space $\HY 3$. 
Thurston [Tu] reinterpreted 
this theorem in terms of patterns of disks on the 2-sphere and he 
observed the existence of disk patterns with preassigned 
overlap angles  not exceeding $\pi/2$
on any 
compact surface $X$. These disk patterns can be interpreted 
as convex subsets of $\HY 3$ which are 
invariant under a group action $\pi_1(X)\times\HY 3\rightarrow \HY3$. 
In this paper we prove the existence and uniqueness of disk patterns 
on compact surfaces with preassigned angles in $]0,\pi[$, 
provided that the system of preassigned angles fulfill an 
additional condition. 
In terms of the corresponding convex subset of $\HY 3$ this condition 
states that the 
extreme points lie 
on the sphere at infinity. 
We prove the existence and uniqueness of disk patterns by
 a refinement of a variatonal method introduced in [Br]. In this process 
we will characterize a disk pattern as a critical point of a 
functional. Furthermore, it will turn out that its critical value 
is the volume of a fundamental domain of the corresponding 
convex subset of $\HY 3$. 
\end{abstract}

\newpage
\tableofcontents
%\addtocounter{section}{-1}
\section*{Introduction}
\addcontentsline{toc}{section}{Introduction}
Let $T$ be a cell decomposition of a compact surface $X$ without boundary, $\E$ its set of edges and 
$\theta :\E\longrightarrow \ ]0,\pi[$ a function. Roughly speaking, we study 
the existence and uniqueness of a Riemannian metric 
of constant curvature on $X$ and of a collection of disks such that 
\begin{vversetzt}{3}
\item[-]
the centers of the disks are the vertices of $T$,
\item[-]
two disks whose centers are joined by an edge $e$ meet in 
an angle $\theta(e)$,
\item[-]
if the centers of some disks form the vertices of a cell of $T$, then 
the bounding circles intersect in a point.  
\end{vversetzt}
\smallskip
\label{intro}
For our purpose it is more convenient to consider such collections of disks 
from a slightly different viewpoint (Chapter~\ref{preliminaries}):
Let $\HY 3$ be the hyperbolic 3-space and $\partial\HY 3$ its boundary.
 The metric structure of $\HY 3$ endows $\partial \HY 3$
 with a natural conformal structure such that the isometries of $\HY 3$ correspond  
to the conformal transformations of $\partial\HY 3$. We call an open subset 
 ${\cal O}$ of 
$\partial\HY 3$ a 
homogeneous domain if its 
stabilizer 
$\stH {{\cal O}}$  
in the isometry group of $\HY 3$  acts transitively on ${\cal O}$. 
%A simply connected homogenous domain 
%which is 
%bounded by a topological circle is called a conformal disk. 
Let $\wt X$ be a simply connected homogeneous domain  and $\pi_1(X)$ a discrete 
subgroup of $\stH {\wt X}$ such that 
 $\wt X/\pi_1(X)$ is homeomorphic to $X$. 
Furthermore, let $P$ be a $\pi_1(X)$-invariant subset of $\wt X$ such that $P/\pi_1(X)$ is finite. 
We denote the convex hull of the set $P$ in $\HY 3$ by $\ch P$.
 If $\dim \ch{P}=3$ we call this convex set an \new{ideal $\pi_1(X)$-polyhedron}. 
Every ideal $\pi_1(X)$-polyhedron has an `external representation' as the intersection of the 
halfspaces supporting the facets of $\ch P$.
Since there is a canonical bijection between the set of 
halfspaces of $\HY 3$ and the set of conformal disks in $\partial\HY 3$, 
%(i.e. simply connected homogeneous domains bounded by topological circles)  
this external representation yields a 
 $\pi_1(X)$-periodic collection of conformal disks in 
$\wt X$. Projecting these disks to $\wt X/\pi_1(X)$, we get a finite collection of 
`disks' in $\wt X/\pi_1(X)$. 
The combinatorics of this projected disk collection is described by a 
cell decomposition of $X$ such that each vertex
 corresponds to a disk  and 
each {\cell} corresponds to a point of the set $P/\pi_1(X)$. 
The dihedral angles of the $\pi_1(X)$-polyhedron can be considered as weights 
on the edges of the cell decomposition. 

Provided that $\theta$ fulfills a necessary condition, we prove the existence of an ideal $\pi_1(X)$-polyhedron whose combinatorics  
correspond to $T$ and whose dihedral angles 
are described by the weight function $\theta$ (Chapters 2 - 4). 
This ideal  $\pi_1(X)$-polyhedron is unique up to isometries of $\HY 3$ and  
 yields a simply connected 
homogeneous domain $\wt X$ in such a way that 
 the conformally flat surface $\wt X/\pi_1(X)$ is uniquely determined by $(T,\theta)$. 

In Chapter~5 finally, we express the volume of a fundamental domain 
of a $\pi_1(X)$-polyhedron in terms of the corresponding disk collection.

\Paragraph{Preliminaries}\label{preliminaries}
\abschnitt{A Model for hyperbolic $3$-space \protect\boldmath $\HY 3$.}\label{metric}
A sphere $S$ in $\HY 3$ endowed with the induced 
 Riemannian metric is a surface of positive constant curvature.
 With geodesic rays perpendicular to $S$, we 
can export this metric to the boundary $\partial\HY 3$. 
Two such metrics on $\partial \HY 3$ are similar if and only if the corresponding 
spheres have the same center. Hence, we get a model characterizing the points of 
$\HY 3$ as similarity classes of Riemannian metrics. In this section we will briefly show how to 
express the attributes and objects of $\HY 3$ in this model.

Let $\CS$ be the standard conformal 2-sphere
and $\CO$ its group  of conformal automorphisms, i.e. $\CS$ is conformally 
equivalent to ${\bf P}^1{\bbb C}$ and the subgroup 
of orientation preserving elements of $\CO$ is isomorphic to 
${\bf PSL}(2,{\bbb C})$. 
A topological circle in $\CS$ is called 
 a \new{ conformal circle} if it is the set of 
fixed points of an orientation reversing involution in $\CO$. A connected 
component of the complement of a conformal circle is called a (conformal) \new{open disk}.
 %A (conformal) \new{disk} is the closure of a (conformal) open disk. 

We call a non-empty connected open subset ${\cal O}$ of $\CS$ a \new{homogeneous domain} if 
its stabilizer $\st {\cal O}$ in 
$\CO$ acts transitively on ${\cal O}$. 
There are three types of simply connected homogeneous domains. Namely, any such domain 
 ${\cal O}$ has the form 
${\cal O}=\CS\setminus A$, where $A$ is either empty, a point or a closed conformal disk. 
Accordingly, ${\cal O}$ equipped with the conformal structure induced by 
$\CS$ is conformally equivalent to $\NS$, the Euclidean plane $\NE$ or 
the hyperbolic plane $\NH$. 
 The group of  conformal automorphisms of ${\cal O}$ is
 just $\st {\cal O}$. 

For every simply connected homogeneous domain which is a proper subset of 
$\CS$ there is only 
 one similarity class of 
complete Riemannian metrics of constant curvature   
having the conformal structure induced by $\CS$. 
We will consider these metrics  as `degenerated' metrics on $\CS$. 
This leads to the following definition: 
An inner product structure $n$ on $\CS$ 
(i.e. an inner product on every fiber of the tangent bundle) is called a \new{singular} 
(Riemannian) metric if there exists a simply connected homogeneous domain
  ${\reg n}$ of $\CS$ 
such that the following two conditions are satisfied: 
\begin{vversetzt}{3}
\item[-]
The restriction of $n$ to ${\reg n}$ is a 
complete Riemannian metric of constant curvature (called the curvature of $n$) 
having the conformal structure induced by \nolinebreak$\CS$. 
\item[-]
For every $q\in \CS\setminus {\reg n}$ 
the inner product of any two tangent vectors at $q$ is \nolinebreak$+\infty$. 
\end{vversetzt}
\smallskip
We call $\reg {n}$
 the \new{regular} domain
 of $ n$ and $\CS\setminus{\reg n}$ 
the \new{singular} domain. 
In the following we consider similar singular metrics 
(i.e. similar on their regular domains) 
as equal and we denote the set of similarity classes of singular metrics by 
$\MS$. We identify every singular metric $m\in \MS$ with a representative of 
constant curvature $-1$, $0$ or $+1$ and 
we define the metric space $(\reg m,m)$ as the set $\reg m$ 
equipped with the metric $m$. The similarities of $(\reg m,m)$ 
(i.e. isometries if the curvature is non-zero) 
 are the elements of the stabilizer $\st m$.

Since $\CO$ maps homogeneous domains to homogeneous domains, $\CO$ acts on $\MS$. 
Corresponding to the type of the regular domain of a singular metric,
 we divide $\MS$ into 
three disjoint subsets: We define $\MR$ (respectively, $\partial \MR$, $\MD$) 
to be the set of all singular metrics whose singular domain is 
empty (respectively, a point of $\CS$, a disk in $\CS$). 
 Note that these three sets are just the orbits of the 
group action $\CO\times\MS\longrightarrow\MS$. 
We call the elements of $\MR$ (respectively, $\MD$) \new{regular metrics}
(respectively, \new{disk metrics}). 
We already mentioned that a metric $m\in\MS\setminus\MR$ is determined by its regular domain. 
Henceforth we will no more distinguish between a metric $m\in\MD$ and the open disk $\reg m$. 
In the same sense we identify a metric $m\in\partial\MR$
 with its singular domain. 
This yields a bijection between $\CS$ and $\partial\MR$. 

Our next task is to equip $\MS$ with a topology: 
A sequence $m^i\in\MS$ converges if there exists a singular 
 metric $m\in \MS$, a compact subset $K\subset\CS$ with 
non-empty interior, and a 
sequence $c_i$ of positive real numbers such that 
\begin{vversetzt}{3}
\item[-]
$K$ is a subset of $\reg m ,\reg {m^i}$ for all but finitely 
many $i$, and 
\item[-]
the restriction 
of $c_i m^i$ to $K$ converges to the restriction of $m$ to $K$. 
\end{vversetzt} 
With the topology given by this definition, $\MR$ and $\MD$ are open sets in $\MS$ bounded 
by $\partial\MR$. Furthermore, the compactification of $\MR$ in $\MS$ is just  $\MR\cup\partial\MR$. 

Since any two conformal 
structures on $\CS$ are conformally equivalent, the group $\CO$ 
acts transitively on $\MR$. 
The stabilizer subgroup of any point $m\in \MR$ is isomorphic to ${\bbb O}(3)$ 
and therefore a maximal 
compact subgroup of $\CO$. 
Hence, $\MR$ is a homogeneous space. 
In fact, we now define a distance function $\dist:\MR\times\MR\longrightarrow {\bbb R}_+$ 
such that $(\MR,\dist)$ is a metric space of constant curvature. 
Let $n,m\in\MR$. If $n_x$ denotes the 
%if $n_x$ denotes the 
inner product induced by $n$ at the point 
$x\in\CS$ there is a smooth function $f:\NS\longrightarrow{\bbb R}_+$ 
such that $m_x=f(x)\cdot n_x$. We define  
$$\dist(m,n):=\frac12\max_{\ x\in\CS} \log\frac{m_x}{n_x}. $$
With this notation  the metric space $(\MR,\dist)$ is 
isometric to the 3-dimensional 
hyperbolic space $\HY 3$ and 
its isometries are the elements of $\CO$. 
 The above compactification then corresponds to
 the Busemann compactification of $\HY 3$. 

Note that the topological 
space 
$\MS$ is homeomorphic to ${\bbb R}^3$. In fact, 
using the Poincar\'e model of $\HY 3$, we first identify $\MR$ with 
the open unit ball ${\bf B}^3$ of ${\bbb R}^3$.  
Let $l$ be a line passing through $0\in{\bbb R}^3$ and define $z,z'$ to be the two 
piercing points of $l$ and $\partial {\bf B}^3$. 
%$\{z,z'\}=\partial {\bf B}^3\cap l$. 
If $x$ is a point on $l$ traveling from 
$0$ towards $z$, then the point $x$ corresponds to a regular metric on the boundary 
$\partial{\bf B}^3=\CS$. If $x=z$ this metric explodes 
at the point $z\in\partial{\bf B}$, i.e. $z$ corresponds 
to a metric with singular domain $\{z\}$.  Continuing the travel along $l$, 
we define $x$ to be a singular metric whose singular domain is a 
closed disk in $\partial{\bf B}^3$ centered at $z$. This disk  
increases if $x$ moves away from $z$ and tends to $\partial{\bf B}^3\setminus \{z'\}$ if $x$ tends to 
infinity. 

Summarizing all these identifications, we get the following diagram
$$\begin{array}{ccccccc}
\MR&\subset&\hspace{-0.0em}\MR\cup\partial\MR\hspace{-0.0em}&\subset&
\MS&\hspace{1em}\supset&\hspace{-1.75em}\MD.
\\
\shortparallel&&\hspace{-0.4em}\shortparallel\hspace{-0.0em}&&
\hspace{-0.4em}\shortparallel&&\hspace{-0.5em}\shortparallel
\\
\,\HY 3&&\hspace{-1.7em}\ \,\HY 3\cup\CS\hspace{-0.4em}&&
{\bbb R}^3&&\hspace{-1em}\{\mbox{conformal disk}\}
\end{array}$$ 

For every disk metric $m$ there is a transformation 
$\varphi_m\in\CO$ 
such that 
$\partial\reg m=\{x\in\CS\mid\varphi_m(x)=x\}$. 
We define $\H m$ as the hyperplane 
$\{x\in\MR\mid\varphi_m(x)=x\}$.
%For a disk metric $m$ we define  to be the hyperplane of 
%$\MR$ whose boundary in $\MS$ coincides with the boundary of 
%the regular domain of $m$. 
Two disk metrics $m$, $n$ are said to \new{overlap} if 
the bounding circles of their regular domains cut in two distinct points, i.e. $\H m\cap\H n\neq\emptyset$. 
 In this case the closure of 
$\reg m\cap\reg n$ is a two-gon in $\CS$. The angles at the two vertices of 
this two-gon coincide. We call this angle the \new{angle enclosed by $m$ and $n$}.
%====================================================
%
\abschnitt{Maps on Compact Surfaces.}
Let $X$ be a compact surface, $\ppi:\wt X\longrightarrow X$ 
its universal covering and 
$\pi _1(X)$ the group of covering transformations. 
Henceforth we consider only compact surfaces which are connected 
and have no boundary. 
Let $\wt T$ be a $\pi_1(X)$-invariant decomposition of $\wt X$ into 
simply connected closed subsets (called the \new{{\cells}} of $\wt T$)
 by some arcs (called the \new{edges} of $\wt T$) 
joining pairs of points (called the \new{vertices} of $\wt T$). 
It is understood that no two edges have a common interior point,
 and that the intersection of two different {\cells} is empty, a vertex 
or an edge. Projecting the {\cells}  of $\wt T$ to $X$ yields 
a decomposition $T$ of $X$ into a finite number of closed subsets. We call 
$T$ a \new{cell decomposition} of $X$ or a \new{map} on $X$, $\wt T$ the \new{lifted} cell decomposition or \new{lifted} map, 
and the projection of a vertex (respectively, edge, {\cell}) of $\wt T$ a 
\new{vertex} (respectively, \new{edge}, \new{{\cell}}) of $T$. 
 We denote the set 
 of vertices (respectively, edges, {\cells}) of $T$
 by $\V$ (respectively, $\E$, $\F$). 
Let $a,b\in \V\cup\E\cup\F$. We call $a$ and $b$ \new{incident} if 
$a\subset b$ or $b\subset a$ and we define
\begin{eqnarray*} \br{a,b}&:=&\left\{\begin{array}{lll}
             1&\mbox{if $a$ is incident to $b$}\\
            0&\mbox{else}.
\end{array}\right.  
\end{eqnarray*} 
The sets  $\wtV$, $\wtE$, $\wtF$ and the incidence relations in $\wt T$ 
are defined in the same manner. 

In the current paper we assume that  every edge of a cell decomposition is
incident to two different vertices - the results can be extended to the more general case without difficulties. 

%
%==========================================
%
\abschnitt{Disk Configurations.}
\label{configuration}
In this section we will give a precise definition of collections of disks 
producing $\pi_1(X)$-polyhedra. 
Let $T$ be a cell decomposition of a compact surface $X$ 
and $\theta : \E\longrightarrow \ ]0,\pi[$ a so called weight function. 
We lift $\theta$ to a function $\wt \theta:\wtE\longrightarrow \ ]0,\pi[$ 
by defining $\wttheta (e):=\theta\cp(e)$. 
A function $\cK:\wtV\longrightarrow \MD$ is called a 
\new{$(T,\theta)$-configuration} if there is a homeomorphism $\Phi$ from $\wt X$ to 
a simply connected homogeneous domain $\reg\cK$ of $\CS$ 
(called the regular domain of $\cK$) such that 
 the following conditions hold:
\begin{versetzt}
\item[{A1)}]
The elements of the group 
$\{\Phi\circ g\circ\Phi^{-1}\mid g\in\pi_1(X)\}$ are restrictions of elements of 
$\CO$, 
\item[{A2)}]
If $g\in\pi_1(X)$ and $g_\Phi:=\Phi\circ g\circ\Phi^{-1}$, then 
$\cK\circ g=g_\Phi\circ \cK$
\item[{A3)}]
If two vertices $v,w$ of $\wt T $ are joined
 by an edge $e$, then the disk metrics $\dual v$ and $\dual w$
 overlap and enclose an angle $\wttheta(e)$, 
\item[{A4)}]
If $v_1,v_2\ldots,v_n$ are the vertices of a {\cell} $f$ of $\wt T$, then 
the circles bounding the regular domains of  ${\dual{v_1}}$, ${\dual{v_2}},\ldots,{\dual{v_n}}$ intersect in 
a point $\dual f$.  
\item[{A5)}]
$\reg\cK$ is the disjoint union of  
$\displaystyle{\!\!\bigcup_{v\in\wtV}\!\!\reg{\dual v}}\ $ and  $\  \displaystyle{\!\!\bigcup_{f\in\wtF}\!\!\dual f}.$
\end{versetzt}
\begin{remarks*}
Henceforth we identify  
$g\in\pi_1(X)$ with  $g_\Phi$ and we consider $\pi_1(X)$ as a subgroup of $\CO$. 
The set $\MD$ is just the set of conformal disks of $\CS$. 
%Every element $m$ of $\MD$ is determined by the open disk $\reg m$. 
Hence, 
a $(T,\theta)$-configuration is a $\pi_1(X)$-equivariant assignment of open disks. 
If $f$ is a cell of $\wt T$ then $\dual f$ is the only point 
 contained in the closure of 
every disk assigned to a vertex incident to $f$. Therefore, the extension of $\Dual$ to the set 
$\wtF$ remains $\pi_1(X)$-equivariant. 
\newline
Since $X\simeq\reg\cK/\pi_1(X)$, the group $\pi_1(X)$ acts properly discontinously on $\reg\cK$. We claim that 
there is always a metric $m\in\MS$ such that $\reg m=\reg\cK$ 
and  $\pi_1(X)$ is a subgroup of the isometry group of $(\reg\cK,m)$. 
In fact, if $\reg \cK=\CS$, then the 
 group $\pi_1(X)$ is trivial or generated by an 
involution according as $X$ is homeomorphic to the sphere or to the 
projective plane. 
Since $\pi_1(X)$ acts without fixed point on $\CS=\partial\MR$, Brouwer's Fixed-Point 
Theorem states that 
there is an element $m\in\MR$ fixed under $\pi_1(X)$. Hence, $\pi_1(X)$ 
is a subgroup of the isometry  group $\st m$ of $(\CS,m)$. 
 If $\reg\cK \neq\CS$ there is a  
metric $m\in\MS$ whose regular domain equals the regular domain of 
$\cK$. Hence, $\pi_1(X)$ is a subgroup of $\st m$. The elements of this 
group are the similarities of the metric space $(\reg\cK,m)$.  
Since every similarity which is not an isometry 
 fixes a point in $\reg\cK$, the elements of $\pi_1(X)$ have to be
 isometries in $(\reg m,m)$.
\end{remarks*}
\abschnitt{Polyhedral Weight Functions.}
\label{PWF}
Our next aim is to develop some necessary conditions for a 
function $\theta:\E\longrightarrow \ ]0,\pi[$ to be the weight function of 
a $(T,\theta)$-configuration. We start with some notation. 
A nonempty ordered family ${\cal F}=(e_1,\ldots,e_n)$ of edges of $T$ is called a \new{chain} 
of edges if there exists a continuous, locally injective path 
\newcommand{\weg}[2]{#1_{#2}}
$\weg{\gamma}{{\cal F}}:[0,n]\rightarrow X$ 
such that for every $i\in\{1,\ldots,n\}$ its  
restriction to the interval 
%Wenn nicht jede Kante zwei verschiedene Ecken hat so muss
%die uebernaechste Zeile durch die naechste ersetzt werden. 
%$]{i-1},i[$ 
$[{i-1},i]$ is injective and  $\weg{\gamma}{{\cal F}}([{i-1},i])=e_i$. 
If $\weg{\gamma}{{\cal F}}$ is a loop, then we call 
${\cal F}$ a \new{loop} of edges. 

A loop ${\cal F}$ of edges is called \new{contractible} if  
$\weg{\gamma}{{\cal F}}$ is contractible. 
A contractible loop of edges is called \new{reduced} if there is no subfamily 
which is a contractible loop of edges. 
Observe, that  the lifts of 
$\weg{\gamma}{{\cal F}}$ are simple closed curves in $\wt X$ if ${\cal F}$ is a  reduced contractible  loop of edges. 
Furthermore, if $f$ is a cell of $\wt T$, then there is a reduced 
contractible loop of edges ${\cal F}$ such that 
$f$ is bounded by a lift of 
$\weg{\gamma}{{\cal F}}$. 

Let $\theta : \E\longrightarrow \ ]0,\pi [$ be a weight function and 
assume that $\cK$ is a $(T,\theta)$-configuration.
Let ${\cal F}=(e_1,\ldots,e_n)$ be a reduced contractible loop of edges, 
$\weg{\wt\gamma}{{\cal F}}$ a lift of $\weg{\gamma}{{\cal F}}$ 
and $v_1,\ldots,v_n$ the vertices of $\wt T$ along the simple closed curve 
$\weg{\wt\gamma}{{\cal F}}$. 
Conditions A3-A5 imply that the circles  
 $\partial\reg{\dual{v_1}},\ldots,\partial\reg{\dual{v_n}}$ have a point in common 
 if and 
only if $\weg{\wt\gamma}{{\cal F}}$ 
 forms the boundary of a {\cell} of $\wt T$. 
In terms of the weight function $\theta$ this can be expressed
 in the following way (Figure~\ref{polyhedral}):
\begin{versetzt}
\item[{B1)}]
$(\pi-\theta(e_1))+(\pi-\theta(e_2))+\cdots+(\pi-\theta(e_n))\geq 2\pi$.
\item[{B2)}]
$(\pi-\theta(e_1))+(\pi-\theta(e_2))+\cdots+(\pi-\theta(e_n))=2\pi$ 
if and only if $\weg{\wt\gamma}{{\cal F}}$ bounds a cell of $\wt T$. 

%the union of the edges 
%$e_1,\ldots,e_n$ is the boundary of all {\cells} of $T$.
\end{versetzt}
\smallskip\indent
We call a weight function $\theta$ of an arbitrary cell decomposition of a 
compact surface \new{polyhedral} if Conditions {B1} and {B2} 
are fulfilled for every reduced contractible loop of edges. %$(e_1,\ldots,e_n)$.

\Figur{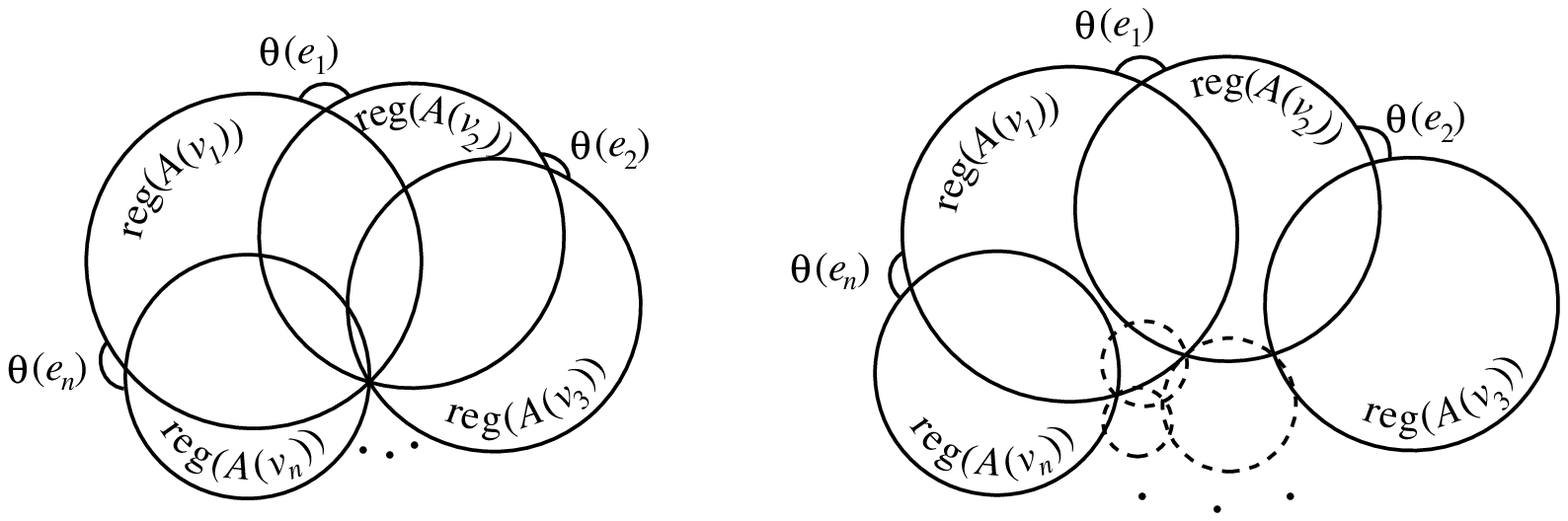}%
{\vspace*{-1em}
Figure~\ref{polyhedral}{\rm{a}}: $\displaystyle{\sum_i\pi-\theta(e_i)=2\pi}$
	\hspace*{2em}
	Figure~\ref{polyhedral}{\rm{b}}: $\displaystyle{\sum_i\pi-\theta(e_i)>2\pi}$\hspace*{2em}}{-1}{5}
%\epsfxsize=0.8\figurwidth

%
%
%=======================================================
%
\begin{remark*}
There are cell decompositions on compact surfaces which do not admit a polyhedral 
weight function. 
As an example consider the truncated tetrahedron,
 i.e. a tetrahedron whose vertices are cut off by planes parallel to the 
opposite {\cell} 
(Figure~\ref{trunc}). 
The surface of this figure admits a cell decomposition $T$
with four triangles and four hexagons. Let $\E$ be the set of edges of $T$, 
$E$ the subset of those edges incident to a triangle and assume that 
$\theta:\E\longrightarrow \ ]0,\pi[$ is a polyhedral weight function. 
Since every edge is incident to a hexagon, we get the following contradiction:
$$4\cdot 2\pi=\sum_{e\in E}\pi-\theta(e)< \sum_{e\in \E}\pi-\theta(e)<4\cdot 2\pi.$$
Therefore, the existence of a polyhedral weight function is 
a combinatorial characteristic of a cell decomposition. 
Nevertheless, on every compact surface there exist numerous cell decompositions 
admitting a polyhedral weight function. As an example we set up 
cell
 decompositions with regular {\cells}, i.e. there is an integer $n$ such that 
 every {\cell}
 is homeomorphic to an Euclidean polygon with precisely $n$ edges. 
The function $\theta:\E\longrightarrow \ ]0,\pi[$, $e\mapsto 2\pi/n$ is then 
polyhedral.  
For more details about polyhedral weight functions of $\NS$ see [Ho].

\vspace*{1ex}
\Figur{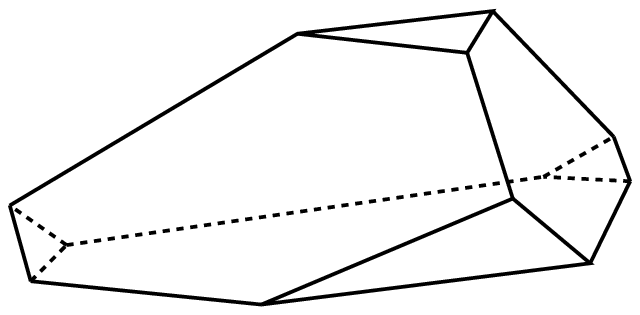}{Figure~\ref{trunc}: Truncated Tetrahedron}{0}{4}
%\vspace*{-2ex}

\end{remark*}
\abschnitt{Disk Packings.}
Let $T$ be a cell decomposition of a compact surface $X$ such that  every {\cell} 
 is homeomorphic to an Euclidean quadrangle. Then the weight function 
$e\mapsto\theta(e)=\pi/2$, $\forall e\in \E$ is polyhedral.  
If $\cK$ is a $(T,\theta)$-configuration and $v$, $w\in\wtV$ are 
incident to a common {\cell} but 
 not joined by an edge, then 
the disks $\reg{\dual v}$ and $\reg{\dual w}$ are tangent. 
Hence, the existence of  a $(T,\theta)$-configuration yields the 
existence of a \new{disk packing}, i.e. a collection of tangent disks. 
For more details about disk packings see [CV]. 
Figure~\ref{packing}b$\ $shows a disk packing on a torus whose 
combinatorics is described by a triangulation. Figure~\ref{packing}a$\ $shows 
the associated $(T,\theta)$-configuration.

\vspace*{3ex}
\Figur{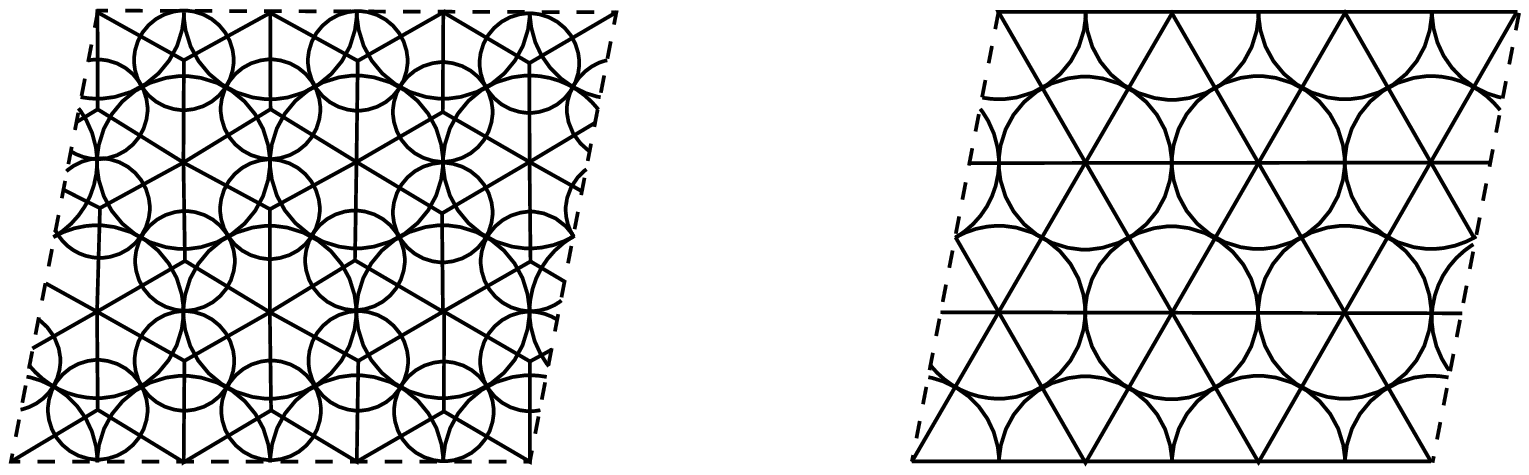}
{\vspace*{-1ex}\hspace{-4em}Figure~\ref{packing}{\rm{a}}:  
	$(T,\theta)$-{\rm configuration}
	\hspace*{4.5em}Figure~\ref{packing}{\rm{b}}: 
	{\rm Disk packing}}{-2.5}{4.8}

\abschnitt{The Main Theorem.}
In this work we establish that Conditions B1 and B2 are sufficient 
for the existence of a $(T,\theta)$-configuration. 
Note that if $\cK$ is a $(T,\theta)$-configuration and $\Phi$ is an element of $\CO$, 
then $\Phi\circ\cK$ is again a $(T,\theta)$-configuration. Hence, the group 
$\CO$ acts on the set of all $(T,\theta)$-configurations. 
We will prove the following:

\begin{theorem*}%
\label{thm}
Let $T$ be a cell decomposition of a compact surface $X$, $\E$ its set of edges
 and $\theta:\E\longrightarrow \ ]0,\pi[$ a 
polyhedral weight function. 
Then there exists a  $(T,\theta)$-configuration which is unique up to 
$\CO$.
\end{theorem*}

\noindent{\bf Remark.}
If every {\cell} of $T$ is homeomorphic to a  triangle or a quadrangle, and  
$\theta(e)\leq\pi/2,\ \forall e\in \E$ the above theorem is a special case of the 
Theorem of Andreev and Thurston ([A1], [A2], [Tu]). If $X$ is homeomorphic to $\NS$ 
compare with the results of Igor Rivin ([R1], [R2]).
\smallskip\newline\indent
Given a cell decomposition $T$ of a compact surface and a polyhedral weight function 
$\theta$ we will construct a convex space $\WC$ and a 
functional $\LaC$ on $\WC$ such that the $(T,\theta)$-configurations 
can be identified with the critical points of $\LaC$. 
The study of this functional leads to the existence and uniqueness 
of   $(T,\theta)$-configurations. 
In the next section we will show that any disk configuration $\cK$ can be 
interpreted as a $\pi_1(X)$-polyhedron $\CH\cK$. If $\psi$ is a 
critical point of $\LaC$ corresponding to the disk configuration $\cK$, then it 
will turn out that  
$\LaC(\psi)$ is the volume  
of a fundamental domain of $\CH\cK$.
\medskip

\abschnitt{The \protect\boldmath Hyperbolic Hull of a Disk Configuration.}
\label{Hull}
Let $T$ be a \mapon a compact surface $X$ and
 $\cK$ a $(T,\theta)$-configuration. 
We define the \new{hyperbolic hull} $\CH {\cK}$ of $\cK$ as 
the smallest subset of $\MR$ such that 
$\CH\cK\cup\{\dual f\mid f\in\wtF\}$ is a closed convex subset of $\MR\cup\reg\cK$
 (with the topology induced by $\MS$). 
Hence, $\CH\cK
%\cup\{\dual f\mid f\in\wtF\}
$ is a $\pi_1(X)$-polyhedron. 
The elements of the set $\{\dual f\mid f\in\wtF\}$ are 
called the vertices of $\CH\cK$. The set of all vertices of $\CH\cK$ is a discrete 
subset in $\reg\cK$. 
If this subset is finite, 
i.e. $\reg\cK=\CS$, then $\CH\cK$ is called a finite ideal polyhedron. 

For every $k\in \ccK$ the intersection of $\H k$ with $\CH\cK$ 
is a closed convex subset of $\MR$. 
%connected 2-dimensional hyperbolic submanifold with boundary. 
We call these sets the \new{facets} of $\CH\cK$. 
If $\reg\cK=\CS$ or $\reg\cK=\CS\setminus\{x\}$, $x\in\CS$, then  
the boundary of $\CH\cK$ in 
$\MR$ is just the union of all facets of $\CH\cK$. If 
$\reg\cK$ is an open disk and $m\in \MD$ with 
 $\reg m=\reg\cK$, then the boundary of $\CH\cK$ in 
$\MR$ is the union of all facets and 
 the hyperplane $\H m$. 

An \new{edge } of $\CH\cK$ finally 
is a non-empty intersection of two different facets of $\CH\cK$.  
The edges of $\CH\cK$ can also be characterized in the following way: 
Let $v$, $w$ be two vertices of $\wt T$ joined by an edge $e$. We define 
$\dual e$ to be 
%the complete geodesic $[\dual v,\dual w]$ i.e. 
the intersection of the
 hyperplanes $\H{\dual v}$ and $\H{\dual w}$. 
With this notation the edges of $\CH\cK$ are just the geodesics $\dual e$, $e\in\wtE$. 

We get a duality between the vertices (respectively, edges, facets) of $\CH\cK$ 
and the {\cells} (respectively, edges, vertices) of $\wt T$. In fact, for every vertex (respectively, edge) 
${\frak e}$ of 
$\CH\cK$ there is one and only one {\cell} (respectively, edge)  $e$ of $\wt T$ with 
${\frak e}=\dual e$ 
and for every facet ${\frak f}$ of $\CH\cK$ there is one and only one vertex $v$ of $\wt T$ 
with ${\frak f}\subset \H{\dual v}$. 

The discrete group $\pi _1(X)$ acting on $\CS$ induces a discrete 
group of isometries of $(\MR,\dist)$ and therefore acts properly discontinuously 
on $\MR$. Hence, $\CH {\cK}/\pi _1(X)$ is a hyperbolic orbifold. 
If in addition $\reg\cK\neq\CS$, then the group $\pi _1(X)$ acts freely on $\MR$ and  
$\CH { \cK}/\pi _1(X)$ is a hyperbolic manifold. 
In fact, let $m\in\MS$ such that $\reg\cK=\reg m$ and assume that $p\in\MR$ is fixed under an element $g\in\pi_1(X)$. 
If $m\in\MD$ (respectively, $m\in\partial\MR$) consider the 
geodesic line passing through $p$ and perpendicular to $\H m$ 
(respectively, passing through $p$ and converging to $m$). This geodesic 
is invariant under $g$ and has a limit point in $\reg\cK$. 
Hence, this limit point is fixed by $g$. Since  $\pi _1(X)$ acts freely on $\reg\cK$, the 
element $g$ has to be the identity. 
\Paragraph{A Characterization of Disk Configurations} \label{coh}
Let $T$ be a cell decomposition of a compact surface $X$. 
The aim of this chapter is to characterize $(T,\theta)$-configurations 
as critical points of some functionals. 
By $\Se$ we denote the set of \new{oriented edges}, i.e. 
$\Se:=\{ ({e,v})\in \E\times \V \mid \br{v,e}=1 \}$. 
We say that the edge $e$ and the vertex $v$ are incident to the 
oriented edge $\sect{e,v}$, and vice versa. 
For $s\in \Se$, $v\in \V$
% and $e\in \E$
 we define the bracket:
\begin{eqnarray*} \br{v,s}&:=&\left\{\begin{array}{lll}
             1&\mbox{if $v$ is incident to $s$}\\
%            1&\mbox{if there exists a $e\in \E$ with $(e,v)=s$}\\
            0&\mbox{else}.
\end{array}\right.  
\end{eqnarray*}
Furthermore, let $\proj\cdot:\Se\longrightarrow \E$
 be the canonical projection and $\rev:\Se\longrightarrow \Se$ 
the orientation reversing function, i.e. 
for every $s\in\Se$ the edge $\proj s$ is incident to the oriented edges $s$ and $-s$. 
If $\wt T$ denotes the lifted cell decomposition, then the set 
of oriented edges $\wtS$ of $\wt T$ and the incidence relations are 
defined likewise. 
We extend the covering projection $\ppi$ to the set of oriented edges 
by defining $\ppi(\sect{e,v}):=\sect{\ppi(e),\ppi(v)},$ for all $\sect{e,v}\in\wtS\subset\wtE\times\wtV$. 
%===============================================================
%
\abschnitt{Angular Datum of a Disk Configuration.}\label{param}
Let $\cK$ be a   
$(T,\theta)$-confi\-gu\-ration and 
$v$, $w\in\wtV$ two vertices incident to an edge $e\in\wtE$. If $m$ is 
an element of $\MS$ such that the closure of $\reg{\dual v}\cup\reg{\dual w}$ 
is contained in $\reg m$, then the conformal disks 
$\reg{\dual v}$ and $\reg{\dual w}$ are metric disks in $(\reg m,m)$, 
i.e. for $k\in\{\dual v, \dual w\}$ 
there is a point $\Mi m k\in \reg k$ and a number $\varrho_m(k)\in {\bbb R}_+$ 
such that  
$$\reg k=\{z\in\reg\cK\mid |z-\Mi m k |_m<\varrho_m(k)\} .$$
We call $\Mi m k$ the \new{$m$-center} of the disk metric $k$. 
%By means of the metric $m$ we can  therefore distinguish a point 
%$\Mi m {\dual v}\in\reg{\dual v}$ and a point  
%$\Mi m {\dual w}\in\reg{\dual w}$.
If 
$f$, $g\in\wtF$ are the {\cells} incident to $e$, then  
we define $Q_m(e)$ to be the geodesic quadrangle in $(\reg m,m)$ with 
vertices $\Mi m {\dual v}$, $\dual f$, $\Mi m {\dual w}$, $\dual g$ and   
 contained in the closure of ${\reg{\dual v}}\cup{\reg{\dual w}}$
 (Figure~{\ref{angular_data}}a). 
The geodesic line through 
$\Mi m {\dual v}$ and $\Mi m {\dual w}$ cut $Q_m(e)$ in two congruent 
triangles (Figure~{\ref{angular_data}}b).

\noindent
We denote the congruence
 class of these triangles by $\tr m( e)$. 
%Up to similarity the triangle $\tr m( e)$ is determined by its angles.  
If $s$ is the oriented edge of $\wt T$ incident to
 the vertex $v$ and the edge 
$e$, then we define $\psi_m(s)$ to be half the
 angle of $Q_m(e)$ at the vertex $\Mi m {\dual v}$. 
Thus, $\psi_m(s)$ and $\psi_m(\rev s)$ are two angles of the 
triangle $\tr m(e )$. The third angle does not 
depend on the metric $m$. It is just $\pi-\wttheta(e)$.
\newline
\Figur{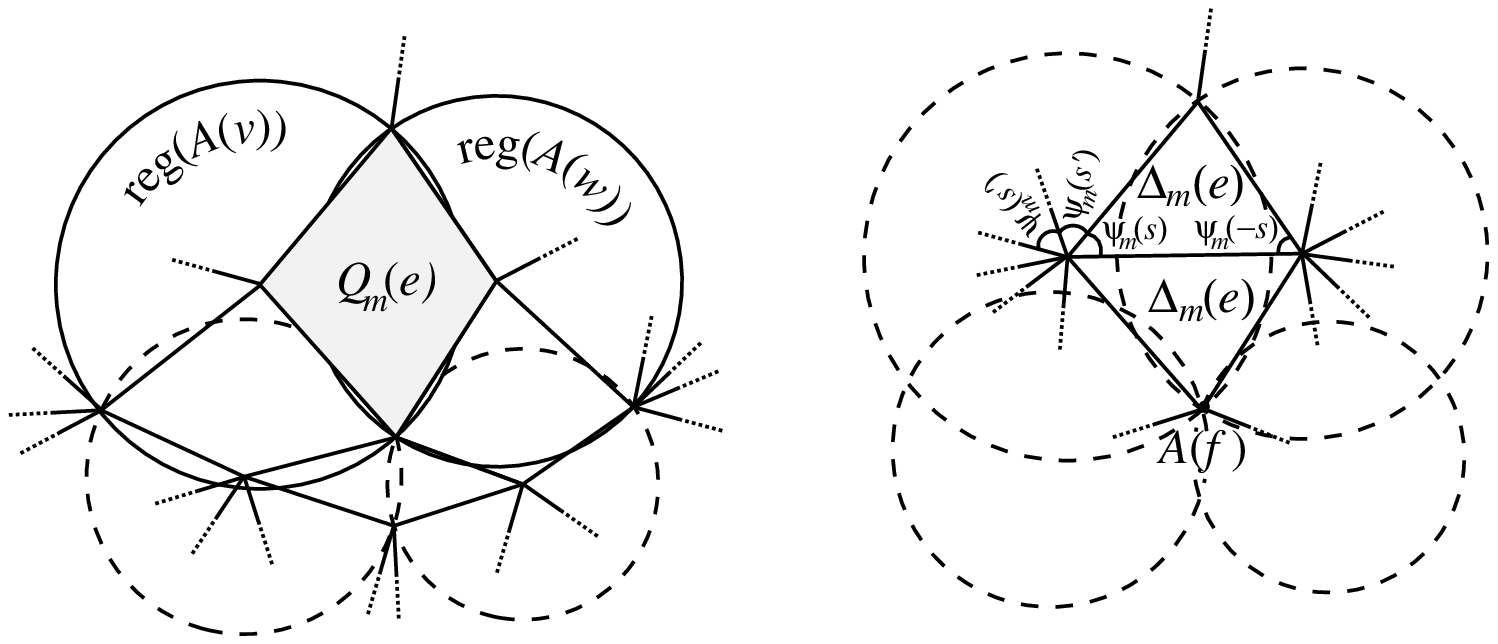}
{\hspace*{0em}Figure~\ref{angular_data}{\rm a}\hspace{11em}Figure~ \ref{angular_data}{\rm b}}{0}{6} 

Our next goal is to describe the configuration ${ \cK}$ by a set of numbers. 
%In~\ref{configuration} we showed that there exists a metric
Let $ m \in \MS$ such that the regular domain of 
${ \cK}$ equals the regular
 domain of $m$ and 
the elements of $\pi_1(X)\subset \CO$ are isometries in $(\reg\cK,m)$. 
The set of quadrangles $\{Q_m(e)\mid e\in\wtE\}$ is a 
decomposition of $\wt X$ 
(Figure~\ref{angular_data}{}a). 
Since $\Dual$ is $\pi_1(X)$-equivariant and 
the elements of $\pi_1(X)$ are isometries in  $(\reg\cK,m)$, this 
decomposition is $\pi_1(X)$-invariant. 
Hence, if $\{\Delta_m\}$ denotes the set of congruence classes of geodesic 
triangles relative to the metric $m$, then the functions 
$$
\begin{array}{ll}
\tr m:\E\longrightarrow {\{\Delta_m\}},&\tr m(e):=\tr m(\ppi^{-1}(e))\\
\psi_m:\Se\longrightarrow {\bbb R},&\psi_m(s):=\psi_m(\ppi^{-1}(s))
\end{array}
$$
are well defined. We call $\psi_m:\Se\longrightarrow {\bbb R}$ the 
\new{angular $m$-datum} of ${ \cK}$.

Let $\psi:\Se\longrightarrow{\bbb R}$ be the angular $m$-datum of a 
$(T,\theta)$-configuration $\cK$, $v$ a vertex of $\wt T$ and 
$s_1,\ldots,s_n$ the oriented edges incident to $v$. 
 The $m$-center of $\dual v$ is a vertex of the decomposition 
$\{Q_m(e)\mid e\in\wtE\}$ 
and 
 $Q_m(\proj {s_1}),\ldots,Q_m(\proj {s_n})$ are just the {\cells} 
incident to this vertex.  
Thus, 
$$\sum_{i=1}^n2\psi(s_i)=2\pi.$$ 
For a function $\psi:\Se\longrightarrow {\bbb R}$ we get 
the following two necessary conditions to be the angular $m$-datum of a 
 $(T,\theta)$-configuration:
\medskip
\begin{versetzt}
\item[C1)]
For every $s\in \Se$ there is a non-degenerate geodesic triangle in the metric space $(\reg m,m)$ with 
 angles $\psi(s)$, $\psi(\rev s)$ and $\pi-\theta(\proj s)$,
\medskip
\item[C2)]
$\displaystyle 
\!\!\!
 \sum_{s\in \Se}\br{v,s}\psi(s)=\pi,\qquad\forall v\in \V.$
\end{versetzt}
\medskip
In the next section we will analyze Condition {C1} in more detail, 
i.e. we examine under what conditions there exists a 
non-degenerate geodesic 
triangle  in  $(\reg m,m)$ with  
prescribed angles $\alpha$, $\beta$, 
{\nolinebreak$\gamma$}. 
%
%===============================================================
%
\abschnitt{Geodesic Triangles.}
%\label{existence}
For a metric $m\in \MD\cup\partial\MR$ any three numbers \\
$\alpha,\beta,\gamma\in{\bbb R}$ fulfilling the inequalities
\begin{eqnarray*}
\alpha,\beta,\gamma&\in&]0,\pi[, \\
\pi-\alpha-\beta-\gamma&\in&\hspace{-0.8em}  \left\{\hspace{-1ex}
\begin{array}{ll}
{]0,\pi[}&\hspace{-1ex}\mbox{if }m\in\MD,\\
{[0,0]}&\hspace{-1ex}\mbox{if }m\in\MR,
\end{array}\right.\\
\end{eqnarray*}
define a similarity class of geodesic triangles in 
the metric space $(\reg m,m)$ and vice versa. 
In this section we will show that an equivalent 
statement is true  for a metric 
$m\in\MR\cup\partial\MR$.

\begin{lemma}\label{existence}
Let $m\in\MS$ and let $c_m\in\{-1,0,+1\}$ be the curvature of $m$. 
\begin{versetzt}
\item[{\sc a)}]
Assume that $d$ is a non-degenerate geodesic triangle in $(\reg m, m)$ with angles 
$0<\alpha,\beta,\gamma<\pi$ 
and define  
$$\ex:=\frac12\left(\alpha+\beta+\gamma-\pi\right),\ \ \ 
\alp:=\alpha-\ex,\ \ \
\bet:=\beta-\ex,\ \ \
\gam:=\gamma-\ex.$$
Then
\begin{eqnarray}
\alp,\bet,\gam&\in\quad&]0,\pi[,\label{tri1}\\\label{tri2}
\ex=\pi-\alp-\bet-\gam&\in\quad&\hspace{-0.8em}  \left\{\hspace{-1ex}
\begin{array}{ll}
{]0,\pi[}&\hspace{-1ex}\mbox{ if }\ c_m=\ \  1 ,\\
{[0,0]}&\hspace{-1ex}\mbox{ if }\ c_m =\ \  0,\\
{]-\pi,0[}&\hspace{-1ex}\mbox{ if }\ c_m= -1 .\\
\end{array}\right.
%\alp+\bet+\gam+\ex&=&2\pi.\nonumber
\end{eqnarray}
\newline
\item[{\sc b)}] 
If $m\in\MR\cup\partial\MR$ and $\alp,\bet,\gam,\ex\in{\bbb R}$ 
are four numbers fulfilling (\ref{tri1}) and (\ref{tri2}), then  
there is a non-degenerate geodesic triangle 
in $(\reg m, m)$ with angles 
$\alpha=\alp+\ex$, $\beta=\bet+\ex$ 
and $\gamma=\gam+\ex$.
\end{versetzt}
\end{lemma}

\begin{beweis}
If $m\in\partial\MR$, then the assertions holds since $\ex=0$.
 
If $m\in\MD$, then $-2\pi<2\ex=-\area d<0$. Since $\alp =\alpha -\ex$ etc. 
we get $\alp,\bet, \gam>0$. From $\alp+\bet=\pi-\gamma$ etc. we then 
conclude that $\alp,\bet, \gam<\pi$. 

Assume that $m\in\MR$. 
 For $\zeta\in\  ]0,\pi[$ we define a $\zeta$-biangle to be the intersection 
of two closed 
conformal disks which are bounded by 
geodesic lines  in $(\reg m,m)$  and enclose an angle $\zeta$.  
The area of a $\zeta$-biangle is just $2\zeta$. 
\newline 
Let $d$ be a triangle as in  {\sc a}). Since the angles of $d$ 
are smaller than $\pi$, we have $0<\area d=2\ex<2\pi$. 
Let ${\rm bi}(\alpha)\supset d$ be an $\alpha$-biangle with vertices $A$, $A'$ as in 
 Figure~\ref{bian_deg}a. Then $\alpha$, $\pi-\beta$, $\pi-\gamma$ are the 
angles of the triangle $d_\alpha:={\rm bi}(\alpha)\setminus d$ 
and $\area{d_\alpha}=2\alp$. Thus, 
$0<2\alp<2\alpha=\area{{\rm bi}(\alpha)}<2\pi$. 
In the same way we conclude that 
$0<\bet,\gam<\pi$. 
\newline
Let $\alp,\bet,\gam,\ex$ as in {\sc b}). Since 
$0<\alp+\ex<\alp+\bet+\gam+\ex=\pi$ we have $\alpha\in \ ]0,\pi[$
 and in the same way we conclude that  $\beta,\gamma\in\ ]0,\pi[$. 
Let ${\rm bi}(\alpha)$ be an $\alpha$-biangle with vertices $A$, $A'$ and $B$ a point in 
the boundary of ${\rm bi}(\alpha)$ such that $B\neq A$, $B\neq A'$. Furthermore let 
${\rm bi}(\beta)$ be a $\beta$-biangle with vertices $B$, $B'$ as in Figure~\ref {bian_deg}a.
The intersection of ${\rm bi}(\alpha)$ and ${\rm bi}(\beta)$ is a non-degenerate
geodesic triangle $d$ with angles $\alpha,\beta,\gamma'$ and 
$\area d=\alpha+\beta+\gamma'-\pi$. 
Varying the point $B$ in our construction, the area of $d$ runs through the interval 
$]0,\min\{2\alpha,2\beta\}[$. Hence, for every $\gamma'\in{\bbb R}$ such that
\begin{equation}\label{local}
0<\alpha+\beta+\gamma'-\pi<\min\{2\alpha,2\beta\}
\end{equation}
there is a non-degenerate triangle with angles $\alpha,\beta,\gamma'$. 
Since $$0<\ex=\alpha-\alp=\beta-\bet<\min\{\alpha,\beta\}$$
inequality (\ref{local}) holds if $\gamma'=\gamma$ i.e. there is a 
non-degenerate triangle with angles $\alpha,\beta,\gamma$. 
\end{beweis}
%
%--------------------{Trigonometric relations.}
%

At the end of this section we state some trigonometric relations 
which we will need later. 
We keep the above  notation and we denote by $a$ the length of the side
 opposite to $\alpha$. 
%with respect to a representative of $m$ with curvature $\sigma(m)$.
 With this notation the half-side formulas of non-Euclidean trigonometry 
can be written as
\begin{equation}\label{formula}
\tan^2\frac{a}{2\sqrt{c_m }}
=\frac{\sin\ex \sin\alp  }{\sin\gam \sin\bet  }.
\end{equation}
%Because $\tan(\sqrt{-1}x)=\sqrt{-1}\tanh x$ 
If $c_m=-1$ this yields 
\begin{equation}\label{formula2}
\tanh^2\frac{a}2
=\frac{\sin({-\ex}) \sin\alp  }{\sin\gam \sin\bet  }.
\end{equation} 
If $d_i$ is a sequence of non-Euclidean triangles converging 
to an Euclidean one, i.e. $\lim_{i\rightarrow\infty}\ex_i=0$, 
$\lim_{i\rightarrow\infty}\alpha_i=\alpha$ etc. these  half-side formulas
reduce to the Euclidean 
law of sines, namely  
$a/b=\sin\alpha/\sin\beta$ (where $b$ denotes the length of the edge 
opposite to $\beta$).

\vspace*{2ex}
\Figur{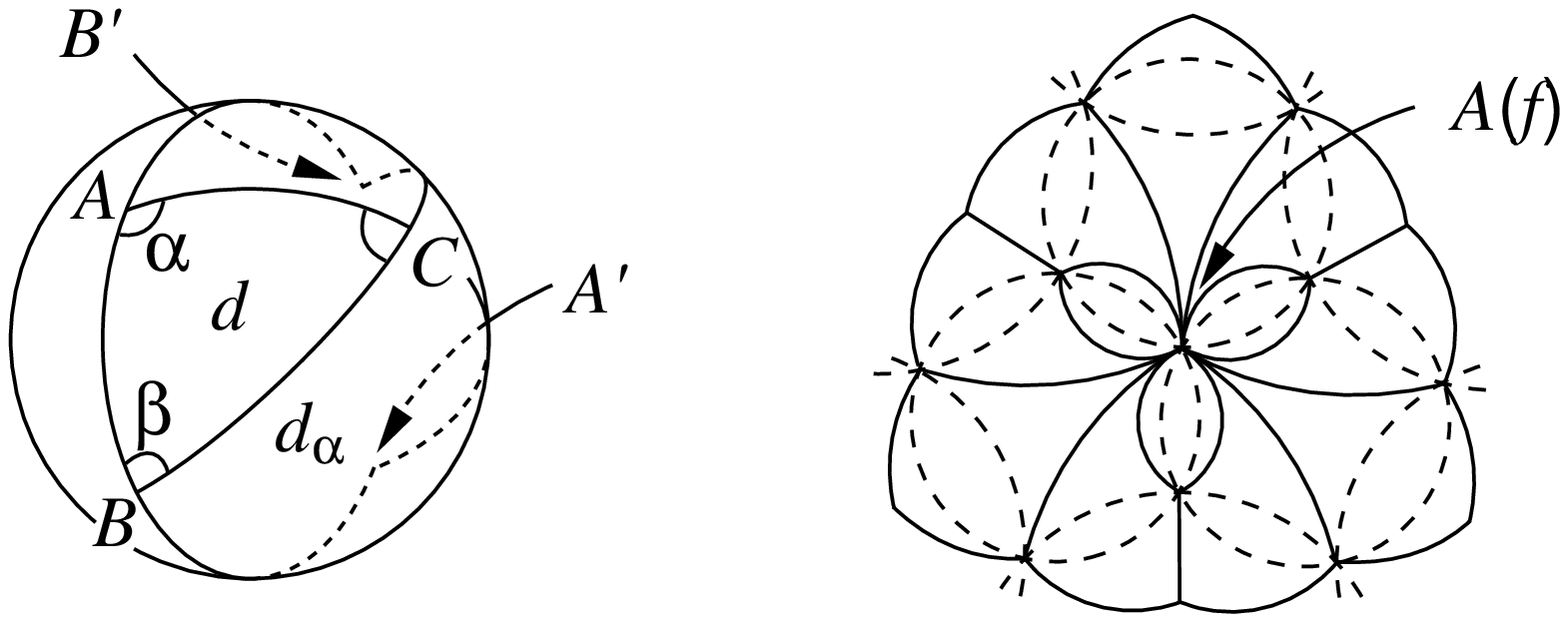}{\vspace*{-0.7em}
\hspace*{-4em}Figure~\ref{bian_deg}{\rm a}\hspace{11em}
Figure~\ref{bian_deg}{\rm b}
}{-1}{5}
%
%===============================================================
%
\abschnitt{Coherent Angle Systems.}\label{WC}
In this section we define 
a set $\WC(T,\theta)$ containing all functions $\psi: \Se\longrightarrow {\bbb R}$ 
fulfilling Conditions C1 and C2. Our main tool will be 
Lemma~\ref{existence}{\refthpunkt}  
 Consider an arbitrary function $\psi: \Se\longrightarrow {\bbb R}$.
 In accordance with Lemma~\ref{existence} 
we define  
$$\begin{array}{ll}
\ex:\E\longrightarrow {\bbb R},
&e\longmapsto\frac12\left(
\psi(s)+\psi(\rev s)-\theta(e)\right),
\mbox{ where $s\in\Se$ with $\proj s=e$},\\
\wh\psi:\Se\longrightarrow {\bbb R},&
s\longmapsto\psi(s)-\ex(\proj s),\\
\wh \gamma:\E\longrightarrow {\bbb R},
&e\longmapsto\pi-\theta(e)-\ex(e).
\end{array}$$
Let $s\in\Se$ and $m\in\MS$. If there exists a non-degenerate triangle in the metric space $(\reg m,m)$
 with angles $\psi(s)$, $\psi(-s)$, $\pi-\theta(\proj s)$, then 
the numbers $\wh\psi(s)$, $\wh\psi(-s)$, $\ex(\proj s)$, $\wh\gamma(\proj s)$ 
lie  
in the intervals prescribed by Lemma~\ref{existence}{\refthpunkt} 
Hence, if $\psi$ is the angular datum of a $(T,\theta)$-configuration $\cK$
and $m\in\MS$ such that  $\reg\cK=\reg m$, then Condition C1 together with Lemma \ref{existence} imply that
\begin{equation}
\begin{array}{ll}
\wh\psi(s)\in \hspace{0.9em}]0,\pi[,&\forall s\in \Se, \label{WC1}\\
\wh\gamma(e)\in\hspace{1em} ]0,\pi[,&\forall e\in \E, \label{WC2}\\
\ex(e)\in \left\{\hspace{-1ex}
\begin{array}{ll}
{]0,c_m\pi[}&\hspace{-1ex}\mbox{if }c_m\neq 0,\\
{[0,0]}&\hspace{-1ex}\mbox{if }c_m = 0,
\end{array}\right.&\forall e\in \label{WC3}\E,
\end{array}\end{equation}
where $c_m$ denotes the curvature of the metric $m$.
 Note that $c_m$ is determined by the Euler characteristic $\chi(X)$ of $X$. In  fact, since $\theta$ is polyhedral, we get 
%Condition C2 yields the following equality: 
\begin{eqnarray}
\label{excesssumme}
\hskip-.5em\pi\cdot\chi(X)&=&\pi(\#\V-\#\E+\#\F)\nonumber\\&=&
\sum_{v\in \V}\sum_{s\in \Se}\br{v,s}\cdot\psi(s)-\pi\#\E+\frac12\sum_{f\in \F}
%\ \sum_{\zweierindex{e\in\E}{e\ \mbox{{\tiny incident to}}\ f}}
\ \sum_{e\in\E}\br{e,f}\cdot
(\pi-\theta(e))\nonumber\\&=&
\sum_{s\in \Se}\psi(s)-\pi\#\E+\sum_{e\in \E}(\pi-\theta(e))\nonumber\\&=&
2\cdot\sum_{e\in \E}\ex(e).
\end{eqnarray} 
Since  $2\ex(e)=c_m\cdot\area {\Delta_m(e)}$ if $c_m\neq 0$,    
the curvature of $m$ equals 
the sign of $\chi(X)$. Together with the equalities 
\begin{equation}
\begin{array}{ll}
 \wh\psi(s)+\wh\psi(\rev s)=\theta(\proj s),&\hspace{1em}\forall s\in \Se,\label{rel_1}\\
 \ex(e)+\wh\gamma(e)=\pi-\theta(e),&\hspace{1em}\forall e\in \E\label{rel_2}
\end{array}
\end{equation}
 this simplifies (\ref{WC1})  
 to 
\begin{equation}
\begin{array}{lll}
\wh\psi(s)\in &\hspace{1.25em}]0,\theta(\proj s)[&\forall s\in \Se, \label{eqn320}\\
\ex(e)\in &
\left\{
\begin{array}{ll}
%Ohne die folgenden { } fehlermeldungen !!!!!
 {]0,\pi-\theta(e)[}&\mbox{if}\ \chi(X) >0,\\
 {[0,0]}&\mbox{if}\ \chi(X)=0,\\
 {]-\theta(e),0[}&\mbox{if}\ \chi(X) <0,\\
\end{array}
\right.
% [0,-2\sigma^2(p)\theta(e)+\sigma^2(p)\pi+\sigma(p)\pi],
&\forall e\in \E.\label{eqn321}
\end{array}\end{equation}
We define $\WT(T,\theta)$ as the set 
 of all functions 
$\psi: \Se\longrightarrow {\bbb R}$ satisfying 
(\ref{eqn320})  
  and  $\WC(T,\theta)$ as the subset of those elements of $\WT(T,\theta)$ 
 satisfying Condition {C2}. 
An element of the set $\WC(T,\theta)$ is called a \new{coherent angle system}. 
If there is no danger of confusion, we write 
$\WT$ and $\WC$ instead of $\WT(T,\theta)$ and $\WC(T,\theta)$. 
If $\chi(X)\geq 0$, then Lemma~\ref{existence} states that
 $\WC$ is the set of all functions $\psi: \Se\longrightarrow {\bbb R}$ 
satisfying Conditions C1 and C2. But note that for $\chi(X)<0$ 
a function $\psi\in\WC$ 
 may not  satisfy Condition C1.
Since $\WT$ as well as $\WC$ are subsets of ${\bbb R}^{\#\Se}$, defined by 
linear equations and inequalities, they are both convex. 
%----------------
%----------------
%----------------
\abschnitt{Stereographic Angular Datum.}\label{relative}
Let $T$ be a cell decomposition of  $\NS$ and assume that ${ \cK}$ is a  
$(T,\theta)$-configuration. 
Every metric $m\in\MR$ yields a specific cell decomposition 
$\{Q_m(e)\mid e\in\E\}$ of $\CS$  and therefore a specific angular $m$-datum 
$\psi_m$ of $\cK$. 
Hence, if $m$ moves along a curve in $\MR$, then $\psi_m$ 
moves along a curve in $\WC$. 

Let $f$ be a {\cell} of $T$, $v_1,\ldots,v_n$  the vertices 
 incident to $f$ and assume that a metric $m\in\MR$ converges 
on a geodesic line towards the Euclidean metric $p:=\dual f$. 
 Since for any disk metric $k$ 
the function $\MR\cup\partial\MR\rightarrow\CS$, $m\mapsto\Mi m k$ 
is continuous, the cell decomposition 
$\{Q_m(e)\mid e\in\E\}$ tends to a cell decomposition 
$\{Q_p(e)\mid e\in\E\}$ of $\CS$. In this process the $m$-centers of the disks $\dual{v_1},\ldots,\dual{v_n}$ tend to 
$\dual f$. Hence, every quadrangle $Q_m(e)$  incident to the 
 $m$-center of such a disk (i.e. $e$ is incident to $f$) tends to 
the degenerate quadrangle $Q_p(e)$ in $(\CS\setminus \{p\},p)$. 
Figure~\ref{bian_deg}{\rm b} shows these degenerate quadrangles for a 
specific $(T,\theta)$-configuration. 
It is not difficult to verify that the angles of these degenerate 
quadrangles are determined by the weight function $\theta$. 
Hence, the union of all non-degenerate quadrangles is a polygon  in the Euclidean plane $(\reg p,p)$ whose exterior angles are determined by $\theta$. 
We define the \new{$f$-stereographic angular datum} $\psi\in\clWC$ by 
$s\mapsto\lim_{m\rightarrow p}\psi_m(s)$.   

If 
$s=\sect{e,v}\in\E\times\V$ is an oriented edge and $-s=\sect{e,w}$, then 
$\psi(s)$ and $\ex(e)$ have 
the following properties:
\begin{center}
\begin{tabular}{|c|c||c|c|c|}
\hline
$\br{v,f}$ & $\br{w,f}$ & $\psi(s)$ & $\ex(e)$ & Remark\\
\hline\hline
$1$ & $1$ & $\frac\pi 2$ & $\frac{\pi-\theta(e)}2$ & $Q_p(e)$ degenerate in $(\reg p,p)$\\
$1$ & $0$ & $0$ & $0$  &$Q_p(e)$  degenerate  in $(\reg p,p)$\\
$0$ & $1$ & $\theta(e)$ & $0$  &$Q_p(e)$  degenerate in $(\reg p,p)$\\
$0$ & $0$ & $\in ]0,\theta(e)[$ & $0$  & $Q_p(e)$ non-degenerate in $(\reg p,p)$\\
\hline
\end{tabular}
\end{center}

We define $\WCf$ as the set of all $\psi \in \clWC$ fulfilling the relations prescribed by the above tabular. 
Defining 
\begin{eqnarray*}
V^*&:=&\V\setminus\{v_1,\ldots,v_n\},\\
E^*&:=&\{e\in\E\mid\mbox{ no vertex of $e$ is incident to $f$ }\},\\
S^*&:=&\{s\in\Se\mid \proj s\in E^*\},
\end{eqnarray*}
every $\psi\in\WCf$ can canonically  be identify 
with its restriction to the set $S^*$, i.e.
\begin{equation}
\label{restriction}
\WCf\equiv
\left\{ \psi:S^*\longrightarrow ]0,\pi[ \ \left|
\begin{array}{l}
\sum\limits_{s\in S^*} \br{v,s} \psi(s)=\theta(v), \ \ \ \forall v\in V^*  \\
 \psi(s)+\psi(\rev s)=\theta(\proj s),\ \ \  \forall s\in S^*
\end{array}
 \right. \right \},
\end{equation}
where 
$$\theta(v):= \pi-\hspace*{-2ex}\sum_{e\in\E\setminus E^*}
\hspace*{-1ex}\br{v,e}\cdot\theta(e)
,\quad\forall v\in V^*.$$
Observe that the numbers $\theta(v)$, $v\in V^*$ are positive. This follows 
since $\theta$ is polyhedral.
\abschnitt{A Functional on the Set of Coherent Angle Systems.}
\label{functional}
In this section we construct an smooth functional  on $\WC$. 
Later on, we will see that configurations of disks occur as critical points of
these functionals. 
First, we introduce the following
\new{Lobachevsky Function} $ \lob:{\bbb R} \longrightarrow {\bbb R} $
$$ \lob(x):= - \int_{0}^{x}\log |2\sin \vartheta | d\vartheta.$$
It is quite easily checked that $\lob$ is well defined as the 
integral converges for all values of $x$. The Lobachevsky Function has the following
properties [Mi]:

\begin{proposition*}
\label{prop1}
\begin{versetzt}\item
\item[1.]
$\lob$ is a continuous and odd function.
\item[2.]
$\lob$ is smooth for all $x\in{\bbb R}$ except for $k\pi$, $k\in {\bbb Z}$.
\item[3.]
$\lob$ is $\pi-$periodic.
\item[4.]
For all $z\in {\bbb R}$ we have $\lob(z)=2\lob(\frac z2)+2\lob(\frac \pi 2+\frac z2)$.
\end{versetzt}
\end{proposition*}

\noindent
If $\phi\in (0,\pi)$, then we define a new function  
$\Ino{\phi}:{\bbb R} \longrightarrow {\bbb R}$ by 
\begin{equation}\label{I_definition}
\Ino{\phi}(x):=
\lob\left({x}\right) +\lob\left(\phi-{x}\right)
-2\lob\left(\frac \phi 2\right).
\end{equation}

\begin{proposition*}
\label{prop2}
For every $\phi\in (0,\pi)$ the continuous function $\Ino{\phi}$ is 
 smooth on $(0,\phi)$. Its
 restriction to the interval $[0,\phi]$ is  
strictly concave and non-positive with maximum value 
$\Ino{\phi}(\phi/2)=0$ and minimum value 
$\Ino{\phi}(0)=\Ino{\phi}(\phi)=2\lob(\frac\pi 2+\frac\phi 2)$.
\end{proposition*}

\begin{proof}
The continuity and smoothness follows immediately from the 
above proposition. 
For $x\in (0,\phi)$ we have 
$$\Ino{\phi}'' (x)=\lob''( x)+\lob''(\phi- x)=
-\cot x-\cot\left(\phi- x\right)=
\frac{-\sin\phi}{\sin x \sin\left(\phi- x\right)}<0.$$
The derivative of $\Ino{\phi}$ vanishes
 if  $x=\phi/2$. Since $\Ino{\phi}$ is concave 
on $[0,\phi]$ the point $x$ is a maximum and the minima lie on the boundary 
of the interval.
\end{proof}

\noindent
We define a function 
$\LaT:\clWT\longrightarrow{\bbb R}$ by 
\begin{equation}
\label{L_definition}
\LaT(\psi):=\frac 12\sum_{s\in \Se}\Ino{\theta (\proj s)}(\wh\psi(s)) -
\sum_{e\in \E}\Imi{\theta(e)}(\ex(e)).
\end{equation}
The above propositions together with (\ref{rel_1}) %and (\ref{rel_2})
 imply that 
$\LaT$ is continuous on $\clWT$, smooth on $\WT$ and 
fulfills the relation 
\begin{equation}\label{normalform}
\LaT(\psi)=\sum_{s\in \Se}\lob({\wh\psi(s)})-
\sum_{e\in \E}\lob\left({\wh\gamma(e)}\right)-
\sum_{e\in \E}\lob\left({\ex(e)}\right)
-\sum_{e\in \E}\lob\left(\theta(e)\right).
\end{equation}
\newcommand{\vvarkappa}{\psi}
We denote the restriction of $\LaT$ to the set $\clWC$ by $\LaC$. 
If $X$ is homeomorphic to $\NS$ and $f$ is a {\cell} of $T$, then we denote the restriction 
of $\LaC$ to the set $\clWCf$ by $\LaCf$. 
The main result of this chapter is the following:
%----------------------------

\begin{theorem}\label{critical_point}
Let $T$ be a cell decomposition of a compact surface $X$ and $\theta$ a 
polyhedral weight function. 
\begin{versetzt}
\item[1)]
Let $p\in \MS$ such that the curvature of $p$ 
equals the sign of the Euler characteristic of $X$. A point 
  $\psi\in\WC$ is a critical point of the functional 
$\LaC$ if and 
only if there exists a $(T,\theta)$-configuration with 
 angular $p$-datum $\psi$. 
This $(T,\theta)$-configuration is unique up to $\st {p}$.
\item[2)]
Assume that $X$ is homeomorphic to $\NS$ and let $f$ be a {\cell} of $T$. 
A point $\psi\in\WCf$ is a critical point of the functional $\LaCf$ if and only if there is 
a $(T,\theta)$-configuration with  $f$-stereographic angular datum 
$\psi$. This $(T,\theta)$-configuration is 
unique up to $\CO$. 
\end{versetzt}
\end{theorem}

\noindent
The proof of this theorem will be divided into three steps:
\smallskip\newline
{{\sc Step A: Proof of %Theorem~\ref{critical_point}
1) if \protect\boldmath$\chi(X)\neq\mbox {\bf 0}$. }}
\label{nichtnull}
Let $T_\vvarkappa\WC$ denote the tangent space of $\WC$ at the point $\vvarkappa\in\WC$,
 i.e. 
$$T_\vvarkappa\WC=\left\{U:\Se\longrightarrow {\bbb R}
\left|\sum_{s\in \Se}\right.\br{v,s}U(s)=0,\quad
\forall v\in \V\right\}.$$
We look for simple tangent vectors. Let $e_1$, $e_2$ be two 
 different edges of $T$ incident to a  vertex $v$. 
We define $U_{e_1,e_2}\in T_\vvarkappa\WC$ by (Figure~\ref{tang_vekt_kr}a):
$$U_{e_1,e_2}(s)= 
\left\{\begin{array}{rl}
1&\mbox{if}\quad s=\sect{e_1,v},\\
-1&\mbox{if}\quad s=\sect{e_2,v},\\
0&\mbox{else.}
\end{array}\right.$$ 
The set of all these tangent vectors span $T_\vvarkappa\WC$.

 Assume that $\psi\in\WC$ is a 
 critical point of $\LaC$ 
and let $v$, $e_1$, $e_2$ be defined as above. For $i=1,2$ we will use the following notation: 
$\alpha_i:=\psi(\sect{e_i,v})$; $\beta_i:=\psi(\rev{\sect{e_i,v}})$; 
$\gamma_i:=\pi-\theta(e_i)$; $\wh\alpha_i:=\wh\psi(\sect{e_i,v})$ 
etc. (Figure~\ref{tang_vekt_kr}b).
\subsubsection*{}
\Figur{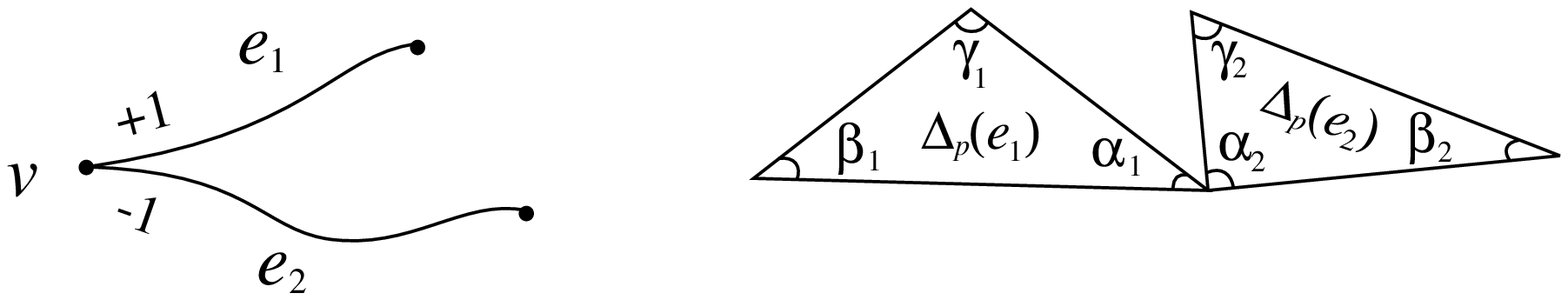}
{\hspace*{-1em}
Figure~\ref{tang_vekt_kr}{\rm{a}}\hspace*{10em}
Figure~\ref{tang_vekt_kr}{\rm{b}}}{0}{2}

Let $(D\LaC)_{\psi}$ denote the tangent map of $\LaC$ at the point $\psi$. 
 Using equation (\ref{normalform}) we obtain:
\begin{eqnarray}
0&=&2\cdot(D\LaC)_{\psi}U_{e_1,e_2}\nonumber\\&=&\label{backward}
-\log\Big|2\sin {\wh\alpha_1}\Big|
+\log\Big|2\sin {\wh\beta_1}\Big|
+\log\Big|2\sin {\ex_1}\Big|
-\log\Big|2\sin {\wh\gamma_1}\Big|\phantom{.}
\\ \nonumber
& &+\log\Big|2\sin {\wh\alpha_2}\Big|
-\log\Big|2\sin {\wh\beta_2}\Big|
-\log\Big|2\sin {\ex_2}\Big|
+\log\Big|2\sin {\wh\gamma_2}\Big|.
\end{eqnarray}
Hence,
\begin{equation}\label{glue}
\frac{\sin\left|{\ex_1}  \right|\sin{\wh\beta_1} }
{\sin{\wh\alpha_1}  \sin{\wh\gamma_1} }=
\frac{\sin\left|{\ex_2}  \right|\sin{\wh\beta_2} }
{\sin{\wh\alpha_2}  \sin{\wh\gamma_2} }.
\end{equation}
Assume for the moment that there exist non-degenerate geodesic
 triangles $\Delta_p(e_1)$, $\Delta_p(e_2)$ in $(\reg p,p)$ 
%with angles $\alpha_i$, $\beta_i$, $\gamma_i$ 
with angles $\alpha_1,\beta_1,\gamma_1$ respectively  $\alpha_2,\beta_2,\gamma_2$. Comparing equation (\ref{glue}) 
with the formulas 
(\ref{formula}) and (\ref{formula2}) %of \ref{existence} 
shows that the legs opposite to the angles $\beta_1, \beta_2$ have the same length. 
But recall that 
we have to prove the existence of these triangles if $p\in\MD$
(see Lemma~\ref{existence}).

Assume therefore that the curvature of $p$ is $-1$. For $i\in\{1,2\}$ 
we will show that $\pi>\alpha_i>0$ and 
$2\ex_i >-\pi$. 
If these inequalities are fulfilled, then there exists a triangle 
in the hyperbolic plane $(\reg p,p)$ with angles 
 $\alpha_i,\beta_i,\gamma_i$.  
Since $\wh\alpha_i\in \ ]0,\pi[$ and $\ex_i\in\ ]-\pi,0[$, we have 
$$\pi>\alpha_i=\wh\alpha_i+\ex_i> -\pi.$$ 
Now suppose that $\alpha_i\leq 0$. Then  
 the following (pairwise equivalent) inequalities hold:
\begin{eqnarray}
% 0&\geq&\sin\alpha_i\nonumber\\
0=\cos\beta_i-\cos\beta_i&\geq&2\sin\alpha_i\sin\gamma_i=
\cos(\alpha_i-\gamma_i)-\cos(\alpha_i+\gamma_i),\nonumber\\
\cos\beta_i+\cos(\alpha_i+\gamma_i)&\geq&
\cos\beta_i+\cos(\alpha_i-\gamma_i),\nonumber\\
\sin\left|{\ex_i}  \right|\sin{\wh\beta_i} &\geq&
\sin{\wh\alpha_i}  \sin{\wh\gamma_i} \label{dazu}.
\end{eqnarray}
Since $(D\LaC)_{\psi}U_{e_1,e_2}=0$ for any pair of edges $e_1,e_2$ 
incident to  $v$, we conclude from 
(\ref{dazu}) and (\ref{glue}) that 
$\psi(s)\leq 0$ for all $s\in \Se$ with $\br{v,s}=1$. 
Hence, the sum of all $\psi(s)$ with $\br{v,s}=1$ has to be non-positive, 
which contradicts Condition C2. 
Arguing in the same way we see that  $\beta_i>0$. 
If $\alpha_i>0$ and $\beta_i>0$ the above calculation shows that 
$$
\frac{\sin\left|{\ex_i}  \right|\sin{\wh\beta_i} }
{\sin{\wh\alpha_i}  \sin{\wh\gamma_i} } < 1
\quad\quad\mbox{and}\quad\quad
\frac{\sin\left|{\ex_i} \right|\sin{\wh\alpha_i} }
{\sin{\wh\beta_i}  \sin{\wh\gamma_i} } < 1.
$$
Multiplying these inequalities we get
$$\sin{\proj{\ex_i}}<\sin{\wh\gamma_i}=\sin
\left(\gamma_i+{\proj{\ex_i}} \right).$$ 
Hence, $\gamma_i+\proj{\ex_i}<\pi-\proj{\ex_i}$ 
and $2\proj{\ex_i}<\pi-\gamma_i$.

Summarizing, we showed that for every oriented edge $s$ of $T$ 
there exists a non-degenerate triangle in the metric space $(\reg p,p)$
 with angles $\psi(s)$, $\pi-\theta(\proj s)$ 
$\psi(\rev s)$. 
We denote its congruence class by $\Delta_p(\proj s)$
and 
the length of its leg  
opposite to $\psi(s)$ by $\len(s)$. 
If $s_1$, $s_2$ are two oriented edges of $T$ incident to a 
vertex $v$, then $\len(\rev{s_1})=\len(\rev{s_2})$. 
Gluing pairs of these triangles as in \ref{param} we get 
 a $\pi_1(X)$-invariant decomposition 
$\{Q_p(e)\mid e\in\wtE\}$ of $\reg p$ (Figure~\ref{angular_data}a). 
Let $v$ be a vertex of $\wt T$ and let $e_1,\ldots,e_n$ to be the edges 
incident to $v$. Then the quadrangles $Q_p(e_1),\ldots,Q_p(e_n)$ have a  
vertex $v_p$ in common and the legs  
incident to $v_p$ have the same length $\varrho_p(v)$. 
We define $\dual v$ to be the disk metric whose regular domain is the metric disk 
in $(\reg p, p)$ with center $v_p$ and 
and radius $\varrho_p(v)$. 
 The map $v\mapsto \dual v$
 is a $(T,\theta)$-configuration.

Conversely, if $\psi$ is the angular $p$-datum of a 
$(T,\theta)$-configuration, then $\psi$ is an element of $\WC$.
 If the edges $e,e'\in\E$ are incident to a  vertex $v$, 
then $\len(\rev{\sect{e,v}})=\len(\rev{\sect{e',v}})$. Reading (\ref{backward}) backwards,
 we conclude that 
$(D\LaC)_\psi U_{e,e'}=0$. 
 Since these vectors span the tangent space, the point $\psi$ is critical.

Evidently, the angular $p$-datum 
determines a disk configuration up to $\st p$. 
%%%%%%%%%%%%%%%%%%%%%%%%%%%%%
%
\smallskip\newline
{{\sc Step B: 
Proof of 1) if \protect\boldmath$\chi(X)=\mbox {\bf 0}$. }}
\label{null} Since 
$$\WC=\left\{ \psi:\Se\longrightarrow \ ]0,\pi[ \ \left|
\begin{array}{l}
\sum\limits_{s\in \Se} \br{v,s} \psi(s)=\pi, \ \ \ \forall v\in \V\\
 \psi(s)+\psi(\rev s)=\theta(\proj s),\ \ \  \forall s\in \Se
\end{array}
 \right. \right \},
$$
we have 
$$T_{\vvarkappa}\WC=\left\{ U:\Se\longrightarrow {\bbb R} \left|
\begin{array}{l}
\sum\limits_{s\in \Se} \br{v,s} U(s)=0, \ \ \ \forall v\in \V \\
 U(s)+U(\rev s)=0,\ \ \  \forall s\in \Se
\end{array}
 \right. \right \}.$$
If ${\cal F}=(e_1,\ldots,e_n)$ is a chain of edges (see \ref{PWF}), then we 
define $v_1,\ldots,v_n$ to 
be the vertices 
$\weg{\gamma}{{\cal F}}(1),\ldots,\weg{\gamma}{{\cal F}}(n)$ 
along the curve 
$\weg{\gamma}{{\cal F}}$ and $s_i:=\sect{e_i,v_i}\in\Se$. 
If ${{\cal F}}$ is a loop of edges, then  
we define a tangent vector 
$U_{\cal F}$ in the following way 
(Figure~\ref{tang_vekt_fl}a):
$$U_{\cal F}(s):= \left\{\begin{array}{rl}
1&\mbox{if there exists a $i\in\{1,\ldots,n\}$ with $s=s_i$,}\\
-1&\mbox{if there exists a $i\in\{1,\ldots,n\}$ with $s=\rev{s_i}$}\\
0&\mbox{else.}
\end{array}\right.$$

\Figur{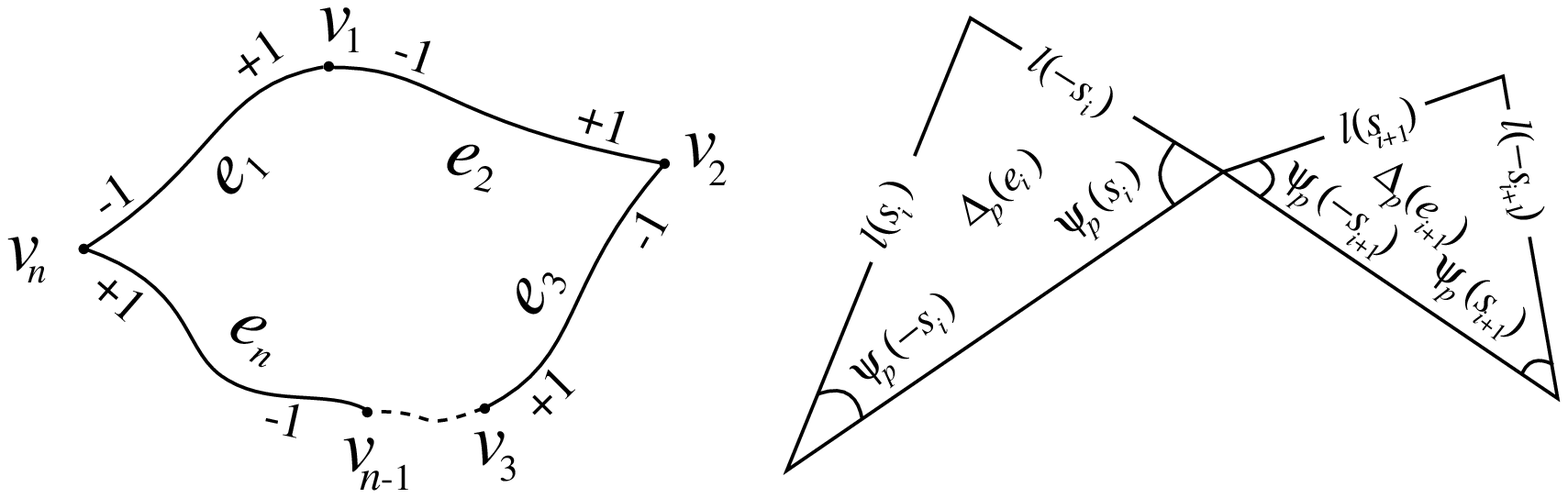}{\hspace*{-4em}
Figure~\ref{tang_vekt_fl}{\rm a}\hspace*{9em}
Figure~\ref{tang_vekt_fl}{\rm b}}{-6}{4}

\noindent
The elements of the set 
$\{U_{\cal F}\mid {\cal F}\ \mbox{is a loop of edges }\}$
 span  $T_\psi\WC$.
 
Let $\psi$ be a critical point of $\LaC$. 
Using  (\ref{normalform})
 the function $\LaC$ reduces to 
\begin{equation}\label{mitterand}
\LaC(\psi)=
\sum_{s\in \Se}\lob\left({\psi(s)}\right).
\end{equation}
For every oriented edge $s$ there is a similarity class of geodesic
 triangles in $(\reg p,p)$ with angles $\psi(s)$, $\pi-\theta(\proj s)$, 
$\psi(\rev{s})$. 
The next step of the proof is to fix a 
 congruence class  $\Delta_p(\proj s)$.  As usual, we denote 
the length of the leg of $\Delta_p(\proj s)$ 
opposite to $\psi(s)$ by $\len(s)$. 
%\newline
We start with an arbitrary edge $e_0\in \E$ and we choose $\Delta_{p}(e_0)$. 
Then for every chain of edges ${\cal F}=(e_1,\ldots,e_n)$ such that $e_1=e_0$ 
 we successively fix $\Delta_{p}(e_2),\ldots, \Delta_{p}(e_n)$ by 
demanding $\len({s_2})=\len(\rev{s_1}),\ldots,\len({s_n})=\len(\rev{s_{n-1}})$ 
(Figure~\ref{tang_vekt_fl}b). 
We claim that the congruence classes $\Delta_p(e_i)$ do not depend on 
${\cal F}$. 
Therefore assume that ${\cal F}=(e_1,\ldots,e_n)$ 
 is a loop of edges and choose 
$\Delta_{p}(e_1),\ldots, \Delta_{p}(e_n)$ as described above. We have to show that 
$\len(\rev{s_n})=\len(s_1)$. 
Since $\psi$ is critical, we conclude from (\ref{mitterand}) that 
\begin{eqnarray*}
0=(D\LaC)_\psi U_{\cal F}=
&+&\!\! \!\!\log|2\sin\psi(\rev{s_1})|-\log|2\sin\psi(s_1)| \\
&+&\!\!\!\!\dots +\log|2\sin\psi(\rev{s_n})|-\log|2\sin\psi(s_n)|.
\end{eqnarray*} 
Hence,
$$1=\frac{
\sin\psi(\rev{s_1}) \cdot\sin\psi(\rev{s_2})
\cdots\sin\psi(\rev{s_n})  }
{\sin\psi(s_1) \cdot\sin\psi(s_2)
\cdots\sin\psi(s_n)  }.$$
Applying the law 
 of sines we get 
$$1=\frac{\len(\rev{s_1}) \cdot\len(\rev{s_2})
\cdots\len(\rev{s_n})  }
{\len(s_1) \cdot\len(s_2)
\cdots\len(s_n)  }
=\frac{\len(\rev{s_n})}{\len(s_1)}.$$
Therefore, the function $e\mapsto\Delta_p(e)$ is well defined. 
Furthermore, if  $s$, $s'\in\Se$ are incident 
to a  vertex $v$, then $\len(\rev s)=\len(\rev{s'})$. 
The remainder of the proof is the same as in Step A. 
%As in \ref{nichtnull} we conclude that there is a tessellation of the 
%regular domain of $p$. Conversely, the existence of such a 
%tessellation implies that $\psi$ is critical.
%
%-----------------
%
\smallskip\newline
{{\sc Step C: 
Sketch of the 
Proof of 2). }}
Using characterization (\ref{restriction}) of $\WCf$,  
every fiber of the tangent space can be identified with the set 
$$\left\{ U:S^*\longrightarrow {\bbb R} \left|
\begin{array}{l}
\sum\limits_{s\in S^*} \br{v,s} U(s)=0, \ \ \ \forall v\in V^* \\
 U(s)+U(\rev s)=0,\ \ \  \forall s\in S^*
\end{array}
 \right. \right \}.$$ 
Let $p\in\partial\MR$ and assume that $\psi$ is a critical point of 
$\WCf$. We will construct a $(T,\theta)$-configuration $\cK$ 
such that $\psi$ is the $f$-stereographic angular datum of $\cK$ 
 and $p=\dual f$. For every $s\in S^*$ 
 we fix a congruence class of triangles $\Delta_p(\proj s)$ 
with angles $\psi(s)$, $\pi-\theta(\proj s)$, $\psi(-s)$ in the same way as in Step B. 
Gluing pairs of these triangles  in the way prescribed 
by $(T,\theta)$ yields 
non-degenerate quadrangles  
 $Q_p(\proj s), s\in S^*$ in $(\reg p,p)$. 
Let ${\cal P}$ be the union of all these quadrangles. The exterior angles of ${\cal P}$ are determined by the weight function 
$\theta$.  
The quadrangles $Q_p(\proj s), s\in S^*$ imply the existence of a disk metric $\dual v$ for 
every $v\in V^*$. On the other hand, 
these disk metrics define the points $\dual{f'}$, $f'\in\F\setminus\{f\}$. 
It remains to define 
 the disk metrics $\dual {v_1},\ldots,\dual {v_n}$. 
Let $w\in\{v_1,\ldots,v_n\}$ and  
 $f,f_1,\ldots,f_k$ the cells incident to 
$w$. The regular domain of $\dual w$ has to be bounded by a circle in 
$\CS$ passing through $\dual f$, $\dual{f_1},\ldots,\dual{f_k}$,
 i.e. a geodesic line $g$ in $(\reg p,p)$ which passes through the points  $\dual{f_1},\ldots,\dual{f_k}$.  
An inspection of the  exterior angles of ${\cal P}$ shows the 
existence of such a geodesic $g$.

Fixing $p\in\partial \MR$ the $f$-stereographic angular datum determines 
a disk configuration with $p=\dual f$ up to $\st p$. 
If we do not fix the point $p$, 
the $f$-stereographic angular datum determines a $(T,\theta)$-configuration 
up to $\CO$.
\endproof
\def\const{{\rm const}}
\Paragraph{Existence and Uniqueness of Disk Configurations}
In this chapter we finally  prove Theorem~\ref{thm}, provided 
that there exist coherent angle systems. 
Their existence  will  be shown in Chapter~\ref{NL}{\refpunkt} 
Our main tool will be Theorem~\ref{critical_point}{\refthpunkt} 
It reduces the 
proof of Theorem~\ref{thm} to a hunt for critical points. %of a functional. 
Henceforth,  $\const$  will denote a number which is constant 
for fixed $(T,\theta)$.  
The proof will be divided into several steps:
%
%-------------
%
\smallskip\newline{\sc Proof of Theorem~\ref{thm} if \protect\boldmath$\chi(X)<\mbox 0$:}
\label{chi<0}
Since $\ex(e)< 0$ for all edges $e$ of $T$, we get   
$$-\Imi{\theta(e)}(\ex(e))=
\Ino{\theta(e)}(|\ex(e)|)
+\const,$$
and 
$$\LaT(\psi)=\frac 12\sum_{s\in \Se}\Ino{\theta (\proj s)}(\wh\psi(s)) +
\sum_{e\in \E}\Ino{\theta(e)}(\proj{\ex(e)})+\const '.$$
Proposition \ref{prop2} yields that 
$\LaT:\clWT\longrightarrow {\bbb R}$ 
is strictly concave. The set $\clWC$ is a convex subset of 
$\clWT$. Therefore, $\LaC$ must be concave, too. 

Since $\WC$ is nonempty (see Chapter~\ref{NL}), the continuous 
and bounded function $\LaC$ has a global maximum 
 $\psi\in\clWC$. If $\psi$ is not a boundary point,  then 
$\psi$ is the only critical point of $\LaC$. 
Assume therefore that the global maximum $\psi$ is a boundary point of 
$\clWC$, i.e. there exists an $s\in\Se$ and an 
$x\in\{0,\theta(|s|)\}$ such that 
$\wh\psi(s)=x$ or $-\ex(|s|)=x$.
Since the function $\Ino{\theta(\proj s)}$ is singular at the point $x$, i.e.  
$$\lim_{x\downarrow 0}\frac\partial{\partial x}\Ino{\theta(\proj s)}(x)=\infty
\qquad \mbox{and}\qquad 
\lim_{x\uparrow \theta(\proj s)}\frac
\partial{\partial x}\Ino{\theta(\proj s)}(x)=-\infty,$$
 the function $\LaC$ decreases if a point tends to the boundary. 
Hence, 
 $\psi\in\WC$ and  
Theorem~\ref{critical_point} states that there exists a $(T,\theta)$-configuration
with angular datum $\psi$.

Assume that $\cK,\cK'$ are two $(T,\theta)$-configurations and let 
$p,p'$ be two metrics such that $\reg p=\reg\cK$, $\reg {p'}=\reg{\cK'}$. 
Since the curvature of $p,p'$ 
equals the Euler characteristic of $X$, the metrics $p$ and $p'$ are elements 
of $\MD$. The group $\CO$ acts transitive on $\MD$. Hence, there is a 
$\Phi\in\CO$ such that 
$\reg\cK=\reg{\Phi(\cK')}$. Theorem~\ref{critical_point} states that 
the angular $p$-data of $\cK$ and $\Phi\circ\cK'$ are both critical points 
of $\LaC$. Since there is only one critical point these angular data coincide,
i.e. 
 there is a $\Phi'\in\st p$ with 
$\cK=\Phi'\circ\Phi\auf{\cK'}$. 
\smallskip\newline
{\sc Proof of Theorem~\ref{thm} if \protect\boldmath$\chi(X)=\mbox 0$:} 
\label{chi=0}
Since $\ex(e)=0$ for all $e\in\E$, the function $\LaT$ reduces to
\begin{equation}\label{chirac}
\LaT (\psi)=\frac 12\sum_{s\in\Se}
\Ino{\theta(\proj s)}\left(\psi(s)\right)+\const.
\end{equation}
Let ${ \cal F}$ be the set of all functions 
$\psi:\Se\longrightarrow {\bbb R}$
 such that $\psi(s)\in \ ]0,\theta(|s|)[$, $\forall s\in \Se$. 
Using (\ref{chirac}) we extend $\LaT$ to the set $\overline{{ \cal F}}$. 
The function $\LaT:\overline{{ \cal F}}\longrightarrow {\bbb R}$ is again 
concave. Since $\LaC$ is just the restriction of $\LaT$ to the 
convex set $\clWC$, the function $\LaC$ has to be convex, too. 
Now, we conclude as in case $\chi(X)<0$.
\smallskip\newline
{\sc Proof of Theorem~\ref{thm} if \protect\boldmath$\chi(X)>\mbox 0$:}
In the above cases a $(T,\theta)$-configuration was unique up to similarity
 if we 
fixed its regular domain. We only had to prove the existence and 
uniqueness of a critical point in $\WC$. If the Euler characteristic 
is positive, i.e. the regular domain is $\CS$, then the proof 
is more delicate for two reasons. 
First, we have no tool to check whether two 
angular data describe the same $(T,\theta)$-configuration 
up to $\CO$ and second, the critical points 
of $\LaC$ are saddle points and no global extremal points. 
We can handle these difficulties by using stereographic angular data. 

Assume first that $\chi(X)=2$ and let $f$ be a {\cell} of $T$. 
We use the notation of~\ref{relative}{\refpunkt} 
The functional 
$\LaCf$ reduces to 
$$
\LaCf(\psi)=\frac 12 \sum_{s\in S^*}
\Ino{\theta(\proj s)}\left(\psi(s)\right)
+\const.$$
Since $\WCf$ is nonempty (see Chapter~\ref{NL}), the 
functional $\LaCf$ takes a global maximum 
$\psi\in\clWCf$. 
As in case $\chi(X)=0$ we conclude that $\LaCf$ is again concave and that
 $\psi$ is the only 
 critical point in $\WCf$. 
Hence, Theorem~\ref{critical_point} states that 
 there is one and only one $(T,\theta)$-configuration 
up to $\CO$.
%Arguing 
%as in \ref{null} we see that $\psi$ is the angular datum of a 
%$(T,\theta)$-configuration relative to the dual vertex of the {\cell} $f$.

If $\chi(X)=1$, then $X$ is covered by $\NS$. 
Since $\wttheta:\wtE\longrightarrow ]0,\pi[$ is polyhedral,  
 there is a 
$(\wt T,\wt \theta)$-configuration ${ \cK}$ which is unique  up to $\CO$. 
Let $g$ be the non-trivial element of $\pi_1(X)$. Then 
$\cK\circ g$ is a $(\wt T,\wttheta)$-configuration, too. Thus, there is 
a $\Phi\in\CO $ such that $\cK\circ g=\Phi\circ\cK$. 
Since $g^2={\rm id}$ the function $\Phi^2$ fixes every element 
of the set $\cK(\wtF)$, i.e.  $\Phi^2={\rm id}$.
\endproof

\bigskip
\Paragraph{Existence of Coherent Angle Systems}\label{NL}
Let $T$ be a \mapon a compact surface $X$ and 
$\theta:\E\longrightarrow\ ]0,\pi[$ a 
polyhedral weight function. In this chapter we will show 
that the set $\WC(T,\theta)$ is non-empty.
 We use a procedure proposed by Yves Colin de Verdi$\grave{\rm e}$re 
[CV] which needs the following theorem of graph theory.
%----------------------------------------
%
%\abschnitt{Compatible Flow Theorem.}

Let $A$ be an antisymmetric relation on the finite set $P$. 
We call the elements of $P$ points and those of $A$ arrows. 
If $(p,q)$ is an arrow, then we call $p$ its initial point and $q$ its endpoint. 
For a set $Z$ of points we denote by $\ra\!\!\!Z\ $  (respectively, $ Z\!\!\!\ra $) 
the set of those arrows having only their endpoint (respectively, initial point) in $Z$. 
A flow $\varphi$ on $(P,A)$ is defined to be a function  
$\varphi: A\longrightarrow {\bbb R}\cup\{-\infty,\infty\}.$ 
%such that the 
%law of Kirchoff is satisfied for every subset $Z\subset P$, i.e. 
%$\!\!\!\!\!\sum\limits_{a\in\ra Z}
%\!\!\!\!\!\varphi(a)=\!\!\!\!\!\sum\limits_{a\in Z\ra}
%\!\!\!\!\!\varphi(a)$.
We will use the following Compatible Flow Theorem 
which can be found in [BE]:

\begin{theorem*}
Let $A$ be an antisymmetric relation on the finite set $P$ and 
$b,B: A\longrightarrow {\bbb R}\cup\{-\infty,\infty\}$ flows on $(P,A)$ 
such that 
\begin{eqnarray*}
&b(a)\leq B(a),&\forall a\in A\nonumber\\
&\sum\limits_{a\in\ra Z}b(a)\leq\sum\limits_{a\in Z\ra}B(a),&\forall Z\subset A\nonumber. 
\end{eqnarray*}
Then there exists a flow 
$\varphi:A\longrightarrow {\bbb R}\cup\{-\infty,\infty\}$ such that 
\begin{eqnarray*}
&b(a)\leq \varphi(a)\leq B(a),&\forall a\in A\\
&\sum\limits_{a\in\ra Z}\varphi(a)=\sum\limits_{a\in Z\ra}\varphi(a),&\forall Z\subset A. 
\label{Kirchoff}
\end{eqnarray*}
We call $\varphi$ a Kirchoff flow compatible with $(P,A,b,B)$.
\end{theorem*}

\noindent
With this theorem we are in a position to prove the existence of 
coherent angle systems. 

\begin{lemma}\label{nichtleer1}
If $\chi(X)\leq 0$, then the convex set $\WC (T,\theta)$ is non-empty.
\end{lemma}

\begin{beweis}
%We use 
%the above notation of \ref{coh}, i.e. $T$ is a \mapon  
%$X$,$\V$ its set of 
%vertices, $\E$ its set of edges and $\F$ its set of {\cells} etc.. 
If $\theta: \E\longrightarrow \ ]0,\pi[$ 
is a polyhedral weight function we have 
\begin{eqnarray}\label{eqn1001}
\pi \cdot\# \V&=&\pi\cdot\# \E-\pi\cdot\#\F+\pi\chi(X)\nonumber\\&=&
\pi\cdot\# \E-\sum_{e\in \E}(\pi-\theta(e))\,+\pi\chi(X) \\&=&
\sum_{e\in \E}\theta(e)\,+\pi\chi(X)\nonumber.
\end{eqnarray}
Consider the finite set $P:=\V\cup \E\cup \{\omega\}$, where $\omega$ is a virtual point, 
together with the relation 
$$A=\Se\cup\left\{(\omega,e)\mid e\in \E\right\}\cup \left\{(v,\omega)\mid v\in \V\right\}.$$
For $\varepsilon >0$, $\tau\leq 0$ we define flows
$b_\varepsilon,B_\tau:A\rightarrow {\bbb R}\cup\{-\infty,\infty\}$ 
by (Figure~\ref{diagramme}a):\medskip
$$\begin{array}{lll}
b_\varepsilon (s)=\varepsilon,& B_\tau (s)=\infty,&\forall s\in \Se \\
b_\varepsilon (\omega,e)=-\infty,& B_\tau (\omega,e)=\theta(e)+\tau,&\forall e\in \E\\
\multicolumn{2}{l}{b_\varepsilon (v,\omega)= B_\tau (v,\omega)=\pi,}
\hspace{1cm}&\forall v\in \V.
\end{array}$$

\Figur{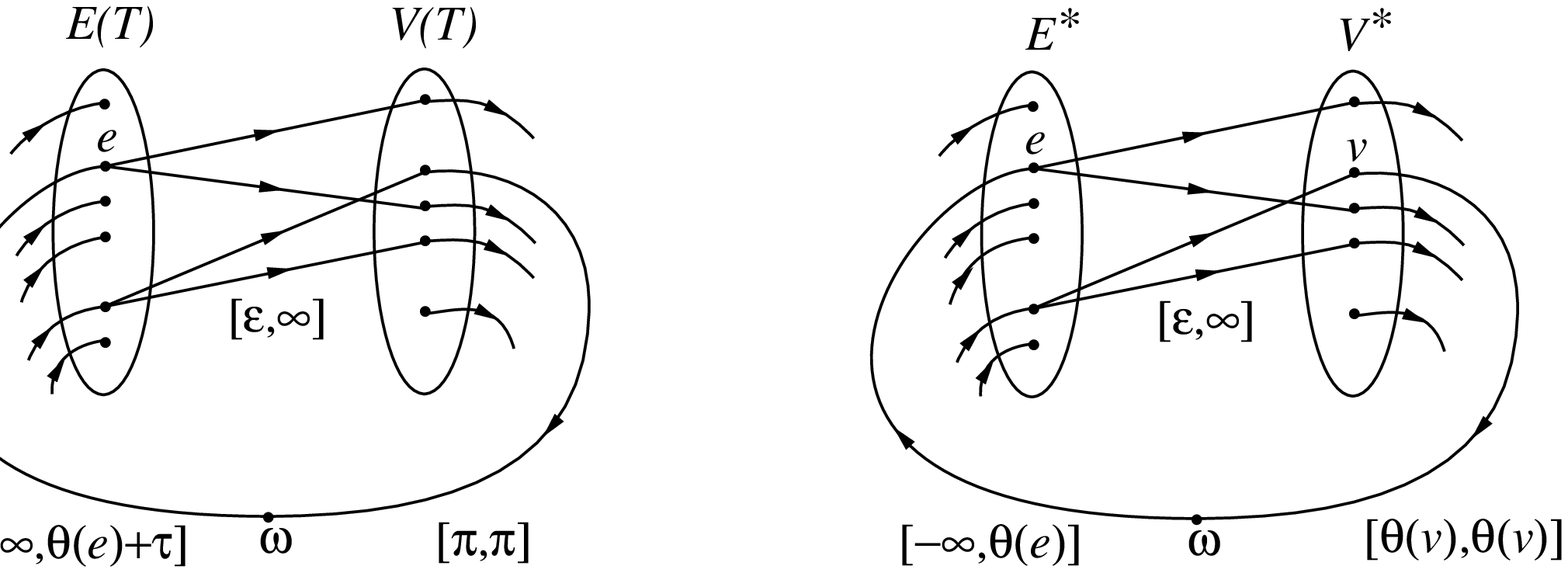}
{\hspace*{0em}Figure~\ref{diagramme}{\rm a}\hspace{11em}Figure~\ref{diagramme}{\rm b}}{-1}{5}

\noindent
If the sign of $\tau$ equals the sign of $\chi(X)$, then we call
 $(P,A,b_\varepsilon,B_\tau)$ a flow diagram of $(T,\theta)$.
 Assume that 
$\varphi$ is a Kirchoff flow compatible with a flow diagram 
$(P,A,b_\varepsilon,B_\tau)$. 
If $\chi(X)<0$, we have $\varphi(\omega,e)<\theta(e)$, for all $e\in\E$. 
On the other hand, if $\chi(X)=0$, then 
Equation (\ref{eqn1001}) together with 
$$\sum_{a\in\ra\{w\}}\varphi(a)=\sum_{a\in\{w\}\ra}\varphi(a)$$
 imply that 
$\varphi(\omega,e)=\theta(e)$ for every $e\in \E$. 
Hence, for every oriented edge $s$ of $T$ 
and every $p\in\MS$ whose curvature equals the sign of $\chi(X)$, 
there is a geodesic triangle in the metric space $(\reg p,p)$ with angles $\varphi(s),\varphi(-s),\pi-\theta(\proj s)$. 
Thus, the restriction of $\varphi$
 to $\Se$ is an element of $\WC(T,\theta)$.

If we can show that for every subset $Z$ of $A$ there is an  $\varepsilon>0$, $\tau\leq 0$ such that 
$(P,A,b_{\varepsilon},B_{\tau})$ is a
 flow diagram 
 of $(T,\theta)$ fulfilling
\begin{equation}\label{eqn1010}
\sum_{a\in Z\ra}B_\tau(a)\geq\sum_{a\in\ra Z}b_\varepsilon(a),
\end{equation}
then the above theorem states that there is a flow $\varphi$ 
compatible with a flow diagram $(P,A,b_{\varepsilon'},B_{\tau'})$. 
%where $\varepsilon:=\min\{\varepsilon(Z)\mid Z\subset A\}$ and 
%$\tau:=\max\{\tau(Z)\mid Z\subset A\}$. 
This implies that $\WC(T,\theta)\neq\emptyset$. 

Let $Z\subset P$, $Z_E:=Z\cap \E$, $Z_V:=Z\cap \V$. By 
$\ek{Z_E\rightarrow \V\setminus Z}$
 we denote the set of those elements of $A$ having its initial point in $Z_E$ and 
its end point in $\V\setminus Z$. We distinguish four cases:

\noindent\makebox
{{\bf\sc Case 1:$\quad\omega\not\in Z,\ Z_E\neq\emptyset$: }}   
 There is a $y\in\ra \!\!\!Z$ with $b_\varepsilon (y)=-\infty$ and  inequality (\ref{eqn1010})
holds for all $\varepsilon,\tau\in{\bbb R}$.

\noindent\makebox
{{\bf\sc Case 2:$\quad\omega\not\in Z,\ Z_E=\emptyset$: }}   
 If $Z_V=\emptyset$ there is nothing to show. Otherwise, we have 
$$\sum\limits_{y\in Z\ra}B_\tau (y)\geq\pi, 
\quad\sum\limits_{y\in \ra Z}b_\varepsilon (y)\leq\varepsilon\cdot\#\V$$ 
and 
(\ref{eqn1010}) holds for  some $\varepsilon >0$.

\noindent\makebox
{{\bf\sc Case 3:$\quad\omega\in Z,\ \ek{Z_E\rightarrow \V\setminus Z}\neq\emptyset$: }}  
  There is a $y\in Z\!\!\!\ra$ with $B_\tau (y)=\infty$.

\noindent\makebox
{{\bf\sc Case 4:$\quad\omega\in Z,\ \ek{Z_E\rightarrow \V\setminus Z}=\emptyset$: }}
 (i.e. if $e\in Z_E$, then both  vertices incident to $e$ are in $Z_V$). 
we have
\begin{eqnarray*}
\sum_{y\in Z\ra}B_\tau(y)&=&\sum_{e\in \E\setminus Z}\!\!\!\theta(e)+
\tau\cdot\#( \E\setminus Z)\\
&=&\sum_{e\in \E}\theta(e)-\sum_{e\in Z_E}\theta(e)
+\tau\cdot\#( \E\setminus Z)\\
&=&\pi\cdot(\#\V-\chi(X))-\sum_{e\in Z_E}\theta(e)
+\tau\cdot\#( \E\setminus Z),
\end{eqnarray*}
where the last equality follows from (\ref{eqn1001}). On the other hand,  
$$\sum_{y\in\ra Z}b_\varepsilon (y)=\varepsilon\cdot\#\ek{\E\setminus Z\ra Z_V}
+\pi\cdot\# (\V\setminus Z).$$
Hence, we have to show that 
\begin{equation}\label{eqn1020}
\pi\cdot\#Z_V-\pi\cdot\chi(X)>\sum_{e\in Z_E}\theta(e).
\end{equation}
Without loss of generality we may assume that for every $v\in Z_V$ there are at least 
two different edges in $Z_E$ incident to $v$. 
In fact, because $\theta(e)<\pi,\quad\forall e\in \E$, 
 inequality (\ref{eqn1020}) holds if we can prove it under this assumption. 

\newcommand{\OLX}[1]{\overline{X_{#1}}}
Let $|Z_E|\subset X$ be the union of all edges in $Z_E$ and 
 $X_1,\ldots,X_n$  the connected components of 
$X\setminus |Z_E|$. 
Starting with the 1-skeleton $|Z_E|$ we reconstruct $X$ in the following way:
 for every $i\in\{1,\ldots,n\}$ we attach a closed surface $\check X_i$ 
(whose interior is homeomorphic to $X_i$) along its boundary 
$\partial \check X_i$. 
Thus, 
$$\chi(X)=\chi(|Z_E|)+\chi(\check X_1)+\cdots+\chi(\check X_n).$$
We may assume that there is an integer $k\in\{0,\ldots,n\}$ such 
that $\check X_{k+1},\ldots,\check X_n$ are the only surfaces homeomorphic to a closed disk, 
i.e. for every $i\in\{k+1,\ldots,n\}$ 
we attach $\check X_i$ along a reduced 
contractible loop of edges. 
Since closed disks are the only surfaces with boundary, connected interior and positive Euler-characteristic, 
we conclude that 
$$\chi(X)\leq\#Z_V-\#Z_E+(n-k).$$
If the edges $e_1,\ldots,e_m\in Z_E$ 
form  a reduced contractible loop of edges in $X$, then
$$\sum_{i=1}^m (\pi-\theta(e_i))\geq 2\pi.$$ 
Hence, 
\begin{equation}\label{eqn1030}
2\sum_{e\in Z_E}\left(\pi-\theta(e)\right)\geq 2\pi (n-k).
\end{equation}
In (\ref{eqn1030}) we have equality 
if and only if 
 $\check X_{1},\ldots,\check X_n$ are the cells of $T$, i.e. 
$Z_E=\E$.  In this case $\ra\!\!\! Z=Z\!\!\!\ra=\emptyset$ 
and we have nothing to show. 
If $Z_E\neq \E$ we get 
$$\sum_{e\in Z_E}\theta(e)<\pi\cdot\left(\# Z_E- (n-k)\right)
\leq \pi\left(\# Z_V-\chi(X)\right).$$
%which is equivalent to (\ref{eqn1020}).
\end{beweis}
%
%---------------------------
%
%
%----------------
%

\begin{lemma}\label{nonempty_positiv}
If $\chi(X)> 0$, then the convex set $\WC(T,\theta)$ is non-empty. Furthermore, 
if $X$ is homeomorphic to $\NS$ and $f$ is a cell of $T$, then 
$\WCf(T,\theta)\neq\emptyset$.
\end{lemma}

\begin{beweis}
First, assume that $\chi(X)=2$, let $f$ be a {\cell} of $T$ and 
$e_1,\ldots,e_n$ the edges incident to $f$. 
Using characterization (\ref{restriction})
 of the set $\WCf(T,\theta)$  we are going to show 
that $\WCf\neq\emptyset$.
\newline
Consider the finite set $P=V^*\cup E^*\cup \{\omega\}$, where $\omega$ is a virtual point, 
together with the relation 
$$A=S^*\cup\left\{(\omega,e)\mid e\in E^*\right\}\cup
\left\{(v,\omega)\mid v\in V^*\right\}.$$
We define bounding flows 
$b_\varepsilon,B_\tau:A\rightarrow {\bbb R}\cup\{-\infty,\infty\}$
 as in the proof of Lemma~\ref{nichtleer1}{\refpunkt} 
The only modifications are $\tau=0$ and 
$b_\varepsilon (v,\omega)=B_0(v,\omega)=\theta(v)$, $\forall v\in V^*$ 
(Figure~\ref{diagramme}b). 

Let $\varphi$ be a flow compatible with $(P,A,b_\varepsilon,B_0)$.  
Using (\ref{eqn1001}) we conclude that 
\begin{eqnarray*}
\sum_{e\in E^*}\theta(e)&=&
\sum_{e\in \E}\theta(e)-\!\!\!\!\!\sum_{e\in \E\setminus E^*}\!\!\!\!\!\theta(e)=\pi\bigl(\underbrace{\# V^*+n}_{\#\V}
-\chi(X)\bigr)-
\!\!\!\!\!\sum_{e\in \E\setminus E^*}\!\!\!\!\!\theta(e)\\
&=&\pi\cdot\#V^*
\underbrace{-\chi(X)\cdot\pi+\sum_{i=1}^n \left(\pi-\theta(e_i)\right)}_{-2\pi+2\pi=0}
\hspace*{1em}-\hspace*{-1em}\sum_{\zweierindex{e\in \E\setminus E^*}{e\not\in\{e_1,\ldots,e_n\}}}
\!\!\!\!\!\theta(e)
=\sum_{v\in V^*}\theta(v).
\end{eqnarray*} 
Hence, the law of Kirchoff at the point $\omega$ implies that 
$\varphi(\omega,e)=\theta(e)$ for all $e\in E^*$, i.e. the restriction of 
$\varphi$ to $S^*$ 
 is an element of $\WCf$.
 Thus, we have to show that for every subset $Q$ of $P$ there is an $\varepsilon >0$ 
satisfying 
\begin{equation}
\sum_{a\in Q\ra}B_0(a)-\sum_{a\in \ra Q}b_\varepsilon (a)\geq 0.
\end{equation}
 We use the same 
notation and we distinguish the same cases as in the proof of  Lemma~\ref{nichtleer1}{\refthpunkt} 
 Only case $4$ 
is a little bit more delicate. 
 We get
\begin{eqnarray*}
\sum_{y\in Q\ra}B_0(y)&=&
\!\!\!\!\!
\sum_{e\in E^*\setminus Q_{E^*}}\!\!\!\!\theta(e)=
\sum_{e\in E^*}\theta(e)-\!\sum_{e\in Q_{E^*}}\!\theta(e)=
\sum_{v\in V^*}\theta(v)-\sum_{e\in Q_{E^*}}\theta(e),\\
\sum_{y\in \ra Q}b_\varepsilon (y)&=&
\varepsilon\cdot\ek{E^*\setminus Q_{E^*}\ra Q_{V^*}}
+\sum_{v\in V^*\setminus Q_{V^*}}\theta(v).
\end{eqnarray*}  
Hence, we have to show that 
\begin{equation}\label{eqn1100}
\sum_{v\in  Q_{V^*}}\theta(v)-
\sum_{e\in Q_{E^*}}\theta(e) >0.
\end{equation}
Let  $\partial{Q_{E^*}}$ be the set of all $e\in\E\setminus E^*$ 
incident to a $v\in Q_{V^*}$. Then
\begin{eqnarray*}
\lefteqn{\sum_{v\in  Q_{V^*}}\theta(v)-
\sum_{e\in Q_{E^*}}\theta(e) =
\pi\cdot\# Q_{V^*}-\sum_
{e\in \partial{Q_{E^*}}}
\theta(e) -\sum_{e\in Q_{E^*}}\theta(e)}\\
& &=\pi\cdot\# Q_{V^*}-\sum_
{e\in \partial{Q_{E^*}}}
\theta(e) -\sum_{e\in Q_{E^*}}\theta(e)+
\underbrace{\sum_{i=1}^n\left(\pi-\theta(e_i)\right)-\pi\chi(X)}_{=0}
\end{eqnarray*}
and (\ref{eqn1100}) reduces to
 \begin{equation}\label{eqn_5000}
\pi\cdot(\#Q_{V^*}+\#\{v_1,\ldots,v_n\})-\pi\cdot\chi(X) >
\sum_{e\in Q_{E^*}}\theta(e)+
\sum_{e\in \partial{Q_{E^*}}}\theta(e) 
+\sum_{i=1}^n\theta(e_i).
\end{equation}
Setting 
$Z_V:=Q_{V^*}\cup\{v_1,\ldots,v_n\}$ and 
$Z_E:=Q_{E^*}\cup\partial Q_{E^*}\cup\{e_1,\ldots,e_n\}$ 
this inequality is just inequality (\ref{eqn1020}) found in the proof of 
Lemma~\ref{nichtleer1}{\refthpunkt}
Since we did not use the restriction $\chi(X)\leq 0$ in the proof of (\ref{eqn1020}), inequality  (\ref{eqn_5000}) holds.
\smallskip\newline\indent
For every $f\in \F$ let $\psi_f$ be an element of $\WCf$. We get an 
element $\psi\in\WC$ 
by 
%if we sum up all these $\psi_f\in\WCf$ with weight $1/\#\F$, i.e.
$$\psi:=\frac{1}{\#\F}\sum_{f\in \F}\psi_f.$$

Now assume that $\chi(X)=1$.
 Then $\wt T$ is a cell decomposition 
of ${\bf S}^2$ and $\wt\theta$ is a polyhedral weight function. 
Hence, there exists an element 
$\psi\in\WC(\wt T,\wt\theta)$. 
Let $g$ be the 
non-trivial covering transformation. If $(e,v)\in\wtS$ 
we define 
$g(e,v)=(g(e),g(v))$.
 We get 
an element $\wb\psi\in \WC(T,\theta)$ by defining 
$\wb\psi(s):=\frac12(\psi\circ g+\psi)(\ppi^{-1}(s))$ for all $s\in \Se$.
\end{beweis}
%
%--------------------------------------------------------
%
\Paragraph{Volume of Ideal Polyhedra}\label{volumen}
The aim of this chapter is to prove the following. 

\begin{theorem}\label{Thm_vol}
Let $T$ be a cell decomposition  of a compact surface $X$.  
%and let $\theta: \E\longrightarrow \ ]0,\pi[$ be a polyhedral weight function. 
If $\psi$ is the angular datum of a $(T,\theta)$-configuration ${ \cK}$, 
then 
$$\vol \left( \CH { \cK}/_{\displaystyle \pi_1(X)}\right)
=\LaT(\psi).$$
\end{theorem} 

\begin{remark*}
Let $X$ be homeomorphic to $\CS$ and $f$ a cell of $T$. 
Since $\LaT$ is continous on $\clWT$, Theorem~\ref{Thm_vol} still holds 
if $\psi$ is the $f$-stereographic angular datum of ${\cal A}$  
\end{remark*}

\abschnitt{Dihedral Angles of Convex Polyhedra.}\label{dihedral}
In this section we will relate the dihedral angles of a
 polyhedron in $\MR$ with angles of 
geodesic polygons. For this purpose we need some preparations: 

Let $\Lambda$ be a closed convex subset of $\MR$. We define the dimension of
 $\Lambda$ as the dimension of the smallest totally geodesic subset of $\MR$ containing 
$\Lambda$. 
A hyperplane ${\cal H}$ of $\MR$ is called a \new{supporting} hyperplane if 
$\Lambda$ is contained in a closed half-space bounded by ${\cal H}$ and 
${\cal H}\cap\Lambda\neq\emptyset$. If ${\cal H}$ is a 
supporting hyperplane such that 
  $\Lambda$ is not contained in ${\cal H}$, then we call 
${\cal H}\cap\Lambda$ a \new{face} of $\Lambda$.
If $\dim\Lambda\geq 1$, a face of dimension $\dim\Lambda-1$ is called a 
\new{facet} of $\Lambda$. 

Let $n\in\MR$. If $m\in\MR$ (respectively, $m\in\MD$), 
then we define $[n,m]$ to be the shortest geodesic segment 
joining $n$ and $m$ (respectively, joining $n$ with a point of the 
hyperplane $\H m$). If $m\in\partial\MR$, then we 
define $[n,m]$ as the geodesic half-line 
starting at $n$ and converging to $m$. We call $[n,m]$ the 
\new{geodesic join} of $m$ and $n$. 
A non-empty 
convex subset $M$ of $\MR$ 
is said to \new{pass} 
through $m\in\MS$ if for every $n\in M$ the geodesic join $[m,n]$ is 
a subset of $M$. 
For arbitrary $m,n\in\MS$ we define the geodesic join 
$[m,n]$ to be the intersection of all geodesic segments, half-lines and lines 
passing through $m$ and $n$.  

Let $M$ be a  subset of $\MR$ and $m\in\MS$. We project $M$ to the boundary 
$\partial\MR$ in the following way: we define $\Lim m M$ to be the set of all 
points $n$ in $\reg m$ such that the intersection of $[m,n]$ and  $M\setminus\{m\}$ is non-empty. 
For a viewer `sited' at $m$, the set $\Lim m M$ is just that 
part of $\reg m$ which is hidden by $M$. 
If $M$ is a hyperplane passing through $m$,
 then there is a unique reflection 
$\Phi \in \CO$ fixing  $M$ pointwise. 
Hence, $\Phi$ fixes $m$, i.e. $\Phi\in\st m$. 
%% We claim that $\Phi$ fixes $m$, too. %obere Linie dann weg
%This is obvious if 
%$m\in\MR\cup\partial\MR$. If $m\in\MD$ let $g,h\subset M$ be two distinct %geodesics passing through $m$. 
%Since $\Phi(g)=g$ and $\Phi(h)=h$, the geodesics $g$, $h$ pass
% through $m$ and $\Phi(m)$. Thus, eihter $\Phi(m)=m$ or $\Phi$ is the 
%reflection in the hyperplane $\H m$. But $g$ is perpendicular to $\H m$ and %therefore $\Phi(m)=m$.   
%\newline
Since the elements of the group $\st m$ are the similarities
 of the metric space 
$(\reg m,m)$, the reflection $\Phi$ is a isometry in $(\reg m,m)$. This 
isometric reflection fixes $\Lim m M=\partial M\cap\reg m$ pointwise. Thus,
 the set   
$\Lim m M$ is a geodesic line in $(\reg m,m)$. Since for any pair 
of points $x,y\in\reg m$ there is a hyperplane passing through $x,y$ and $m$, 
 every geodesic line in $(\reg m,m)$ arises in this way. 

Let $\Lambda$ be a 3-dimensional closed convex subset of $\MR$, 
$m$ a point of  $\in\MS$
and  ${\frak f}_1,\ldots,{\frak f}_n$ the facets of $\Lambda$
 passing through  $m$.  
Furthermore assume that $\Lim m \Lambda$ is a polygon in $\reg m$ 
bounded by $\Lim m {{\frak f}_1},\ldots, \Lim m {{\frak f}_n}$. 
 Since each facet ${\frak f}_1, 
\ldots,{\frak f}_n$ is contained in a hyperplane passing through $m$, 
the segments $\Lim m {{\frak f}_1},\ldots,\Lim m {{\frak f}_n}$ are geodesic segments in 
$(\reg m,m)$. 
Combining this with the fact that every hyperplane intersects $\CS$ 
perpendicularly, we conclude that the angles of the geodesic polygon $\Lim m\Lambda$ 
in $(\reg m,m)$ coincide   
with the dihedral angles of the polyhedron $\Lambda$ at the `vertex' $m$. 
\abschnitt{Decomposition of 
\protect\boldmath$\CH { \cK}$ 
into a Set of Signed Simplices.}\label{signed_simplices}
Let  $M_i\subset\MR, i\in I$ be a family of subsets and
 $\varepsilon_i\in \{-1,0,1\}$, $i\in I$ a family of flags.  
We define the union of the signed sets $\varepsilon_i M_i$ as 
$$
\bigcup_{\{i\in I\mid\varepsilon_i=+1\}}\!\!\!\! M_i \quad\setminus
\bigcup_{\{i\in I\mid\varepsilon_i=-1\}}\!\!\!\! M_i.$$

If  $m$ is a point 
of $\MS$ and $\lambda\subset\MR$ a closed convex set, then   
we define the \new{cone} with \new{base} $\lambda$ and \new{apex} $m$ as
 the smallest convex subset of $\MR$ containing $\lambda$ and passing through 
$m$. We denote this closed convex
 subset of $\MR$ by 
$\Cone m\lambda$. Thus, $\Cone m\lambda$ is just 
 the union of 
all geodesic joins $[m,x],\ x\in\lambda$.

%Cones will serve us as a tool to decompose convex sets. 
We are going to decompose convex sets into signed cones. 
Let $\lambda$ be a facet of a closed convex set 
$\Lambda$ and $m\in\MS$. 
We define an 
index $\Index m\lambda\Lambda$ indicating the sign of $\Cone m \lambda$  by 
$$ \Index m\lambda\Lambda:=\left\{
\begin{array}{rll}
+1&\mbox{if}&\Cone m \lambda\cap\Lambda\neq\lambda,\\
0&\mbox{if}&\dim\Cone m \lambda=\dim\lambda,\\
-1&\mbox{else}.&
\end{array}
\right.$$
If $m\in\MR$ see Figure~\ref{index}, where $\dim\Lambda=2$, 
$H$ denotes the hyperplane carrying $\Lambda$ and the dotted line is the 
geodesic carrying $\lambda$.

\vspace*{2ex}
\Figur{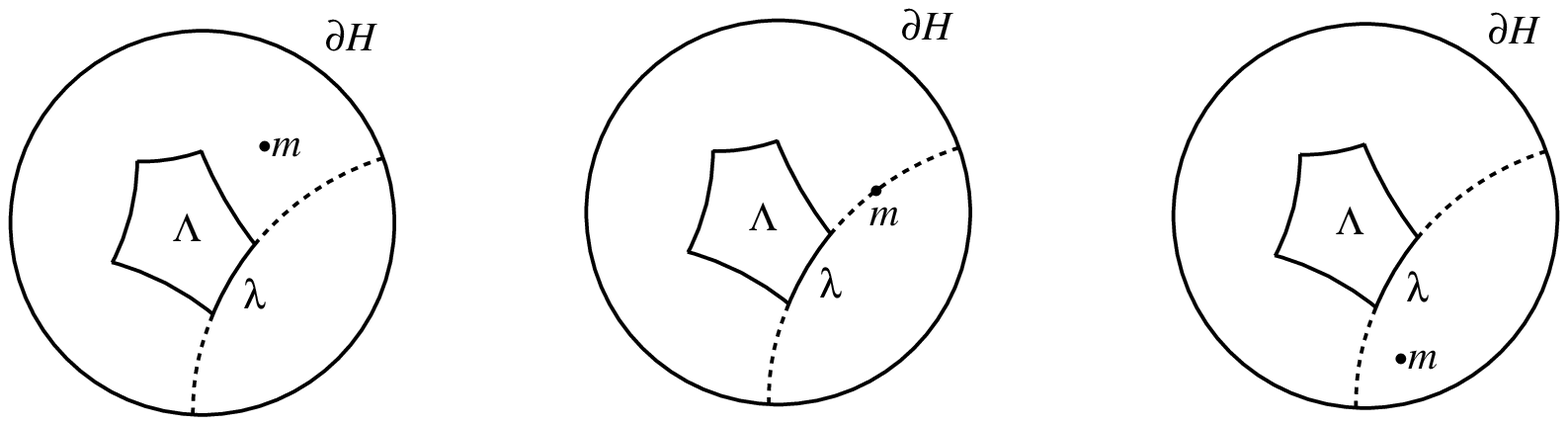}
{\vspace*{-2em}$$\Index{m}{\lambda}{\Lambda}=+1\hspace{3.2em}
\Index{m}{\lambda}{\Lambda}=0\hspace{3.2em}
\Index{m}{\lambda}{\Lambda}=-1
$$
\vspace*{-1em}
\newline
Figure~\ref{index}}{-1}{3}

Let $T$ be a cell decomposition of a compact surface $X$ and $\cK$ a $(T,\theta)$-con\-fi\-gu\-ra\-tion. 
For every $v\in\wtV $ the
intersection of the hyperplane $\H{\dual v}$ with $\CH\cK$ is a facet of the 
set $\CH\cK$. We denote this facet by $\fac{v}$. 
Let $m\in\MS$ such that $\reg m=\reg\cK$ and $m$ is $\pi_1(X)$-invariant. 
We decompose  
$\CH\cK$ in a set of signed cones with apex $m$ and bases $\fac v$, $v\in\wtV$: 
\begin{equation}
\label{pol->cone}
\CH{ \cK}\ =\hspace{-1em}
\bigcup_{\zweierindex{w\in \wtV}{\NorIndex m {\fac w} {\CH\cK}=1}}\hspace{-1em}
\Cone m{\fac w} 
\quad\setminus\hspace{-1em}
\bigcup_{\zweierindex{w\in \wtV}{\NorIndex m {\fac w} {\CH\cK}=-1}}\hspace{-1em}
\Bigr(\Cone m{\fac w} \setminus\fac w\Bigl) .
\end{equation} 
If $m\in\MR$, then 
Figure~\ref{zerlegung} illustrates this decomposition. It shows the intersection of $\CH\cK$ with a hyperplane containing $m$. 
\subsubsection*{}

\Figur{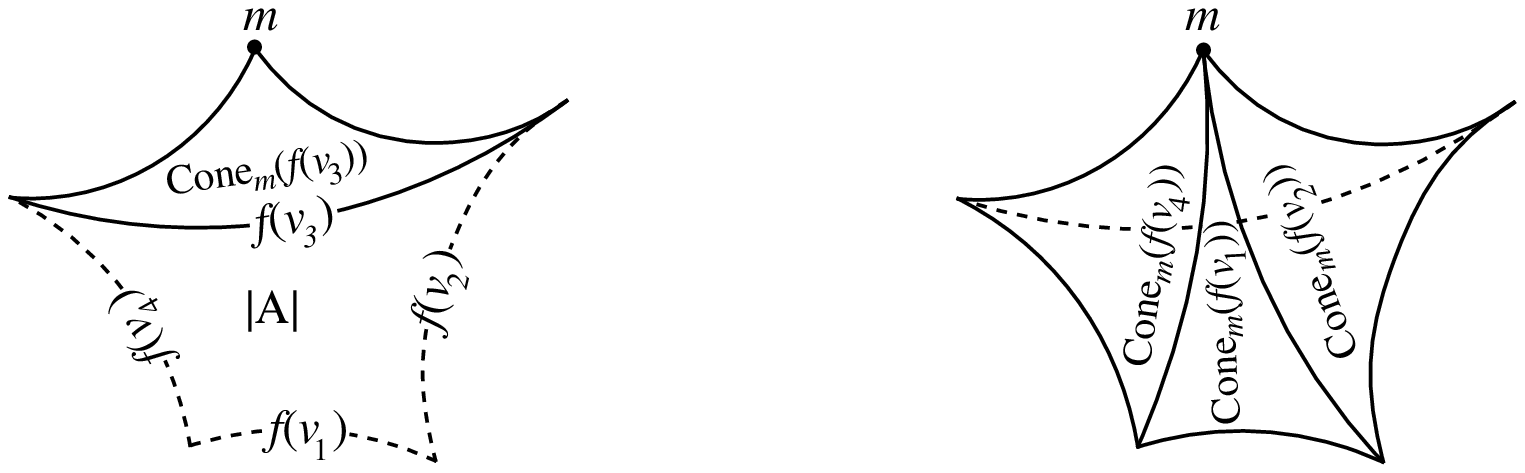}
{\vspace{-2ex}\hspace*{11em}$$\Index m {\fac {v_3}} {\CH\cK}=-1\hspace{7em}
\Index m {\fac {v_i}} {\CH\cK}=1,\ i=1,2,4$$
\vspace*{-1em}
\newline
Figure~\ref{zerlegung}}{0}{4}

\noindent
If $m\in\partial\MR\cup\MD$, then $\Cone m{\fac w}\subset\CH\cK$ 
for any $w\in\wtV$. Hence, 
$\Index m {\fac {w}} {\CH\cK}=1,\quad\forall w\in\wtV$ and 
$$\CH{ \cK}\ =\bigcup_{w\in \wtV}\Cone m{\fac w}. $$

Our final aim is to decompose $\CH\cK$ in a set of signed simplices. For that purpose we first 
decompose the facets $\fac w$, $w\in\wtV$. Let $v$ be an arbitrary but fixed vertex of $\wt T$ and 
$$E(v):=\{e\in E(\wt T)\mid e\mbox{ incident to $v$ }\}.$$ 
The facets of the 2-dimensional closed convex set $\fac v$ are just the geodesic lines $\dual e$, $e\in E(v)$ (see~\ref{Hull}). 
In the same way as we decomposed $\CH\cK$ 
in a set of cones with apex $m$, we now decompose the facet $\fac v$. First we  
project the metric $m$ to the hyperplane $\H{\dual v}$. 
Namely, if $g$ is the geodesic line passing through $m$ and $\dual v$, 
then we define $m(v)$ as the piercing point of $g$ with the hyperplane 
$\H{\dual v}$. We have 
\begin{equation}
\label{facet->triangle}
\fac v =\hspace{-1cm}
\bigcup_{\zweierindex{e\in E(v)}{\NorIndex{m(v)}{\dual e}{\fac v}=1}}
\hspace{-1cm}\Cone{m(v)}{\dual e} 
\quad\setminus\hspace{-1.5em}
\bigcup_{\zweierindex{e\in E(v)}{\NorIndex{m(v)}{\dual e}{\fac v}=-1}}
\hspace{-1cm}
\left(\Cone{m(v)}{\dual e}\,\setminus\,\dual e\right)  .
\end{equation}
We now combine  (\ref{pol->cone}) and (\ref{facet->triangle}).
Let $s\in \wtS$ be incident to $v$. We define 
$$\Fig m s:=\Cone m {\Cone{m(v)}{\dual{\proj s}}}$$
and an index $\ep m s\in\{-1,0,1\}$ by 
$$\ep m s:=\Index{m}{\fac v}{\CH\cK}\cdot\Index{m(v)}{\dual {\proj s}}{\fac v}.$$
Thus, if the symbol `$\approx$' means equality up to a set of measure zero, we get
\begin{equation}\label{pol->sim}
\CH{ \cK}\ \approx\bigcup_{\zweierindex{t\in\wtS}{\ep m t=+1}}
\Fig m t\quad\setminus\bigcup_{\zweierindex{t\in\wtS}{\ep m t=-1}}
\Fig m t.
\end{equation}
%
%
%\Remark
%It is possible to express the index $\varepsilon_m$ purely in terms of $\CS$. 
%In fact, let $s\in \wtS$; $k,l\in { \ccK}$ such that 
%$\br{\invdual k,s}=\br{\invdual l,-s}=1$ 
%and let $K,L$ be the regular domains of $k$ and $l$. If $\partial K$
% is a  geodesic segment in $(\reg m,m)$ 
%we define $\ep m s=0$. 
%Otherwise let $g$ be the geodesic containing $\partial K\cap\partial L$. %Furthermore, let 
%$N$ be the disk with 
%$\partial N\supset g$ and $N\supset\partial K\cap L$. 
%The index $\varepsilon_m$ is then given by 
%$$
%\ep m s=
%\left\{\begin{array}{rl}
%+1&\mbox{if}\quad \Mi m k\not\in N,\\
%0&\mbox{if}\quad \Mi m k \in \partial N,\\
%-1&\mbox{if}\quad \Mi m k \in N\setminus\partial N,
%\end{array}\right.$$
%where $\Mi m k$ denotes the center of the disk $K$ relative to $m$. 

\abschnitt{Proof of Theorem~\ref{Thm_vol}}
Consider decomposition (\ref{pol->sim})  and let $t\in\wtS$. 
Since $m$ is $\pi_1(X)$-invariant 
and $\cK$ is  $\pi_1(X)$-equivariant, we have 
$\ep m t=\ep m{g(t)}$ and $g(\Fig m t)=\Fig m {g(t)}$, 
$\forall  g\in\pi_1(X)$.  
Thus, for $s\in \Se$ the numbers 
$\ep m s:=\ep m{\ppi^{-1}(s)}$, $\vol\Fig m s:=\vol\Fig m {\ppi^{-1}(s)}$ 
are well defined and 
$$\label{vol_dec}
\vol \left(\CH { \cK}/_{\displaystyle \pi_1(X)}\right)
=\sum_{s\in \Se}\ep m s\cdot\vol\Fig m s.$$
Theorem~\ref{Thm_vol} follows now from Lemma~\ref{lem_vol}{\refthpunkt}

\begin{lemma}\label{lem_vol}
Let $\psi_m$ be the angular $m$-datum of a $(T,\theta)$-configuration $\cK$. 
For every $s\in\Se$ the following volume formula holds:
$$
\varepsilon_m(s)\cdot\vol\Fig m s + \varepsilon_m(-s)\cdot\vol\Fig m {-s}=
\Ino{\theta (\proj s)}(\wh\psi_m(s)) -\Imi{\theta(\proj s)}(\ex_m(\proj s))
.
$$
\end{lemma}

\begin{beweis}
\newcommand{\ha}{\dual f}
\newcommand{\hap}{\dual g}
In the remainder of this proof let $s$ be an arbitrary but fixed oriented edge of $T$ 
and $\wt s$ an element of $\wtS$ such that $\ppi(\wt s)=s$.
Furthermore let 
$f,g$ be the {\cells} of $\wt T$ incident to $\proj {\wt s}$, let 
$v,w$  be the vertices of $\wt T$ incident to $\proj{\wt s}$ and define $k:=\dual v$, $l:=\dual w$. 
\smallskip\newline\indent
Assume first, that 
$\Index{m}{\fac v}{\CH\cK}\cdot\Index{m}{\fac w}{\CH\cK}\neq 0$, 
i.e. $m\not\in\H k\cup\H l$.
\newline
The geodesic line $g$ passing through $m$ and $k$ is invariant under the group 
$G:=\st m\cap\st k$. 
 The elements of $G$ are the isometries of $(\reg m,m)$ fixing $\reg k$. 
Since the $m$-center $\Mi m k$ of $k$ is the only point in $\reg k$ 
invariant under  $G$, the geodesic line $g$ has to pass 
through this point. 
An analogous consideration shows that the geodesic line passing through 
 $m$ and $l$ passes also through 
the $m$-center $\Mi m l$ of the disk metric  $l$. 
If $m\in\partial\MR$ we illustrate this in the half-space model 
$\MR={\bbb C}\times{\bbb R}_+$ with boundary $\partial\MR={\bbb C}\cup\{\infty\}$. In this model the geodesic lines passing through 
$\infty$ are the Euclidean half-lines $\{z\}\times{\bbb R}_+$, $z\in {\bbb C}$ 
and the geodesic lines not passing through $\infty$ are Euclidean semi-circles centered at a point $z\in  {\bbb C}\times\{0\}$. If $m=\infty$, then 
Figure~\ref{orthoschem}{a} shows the intersecting hyperplanes 
$\H k$ and $\H l$. 

\Figur{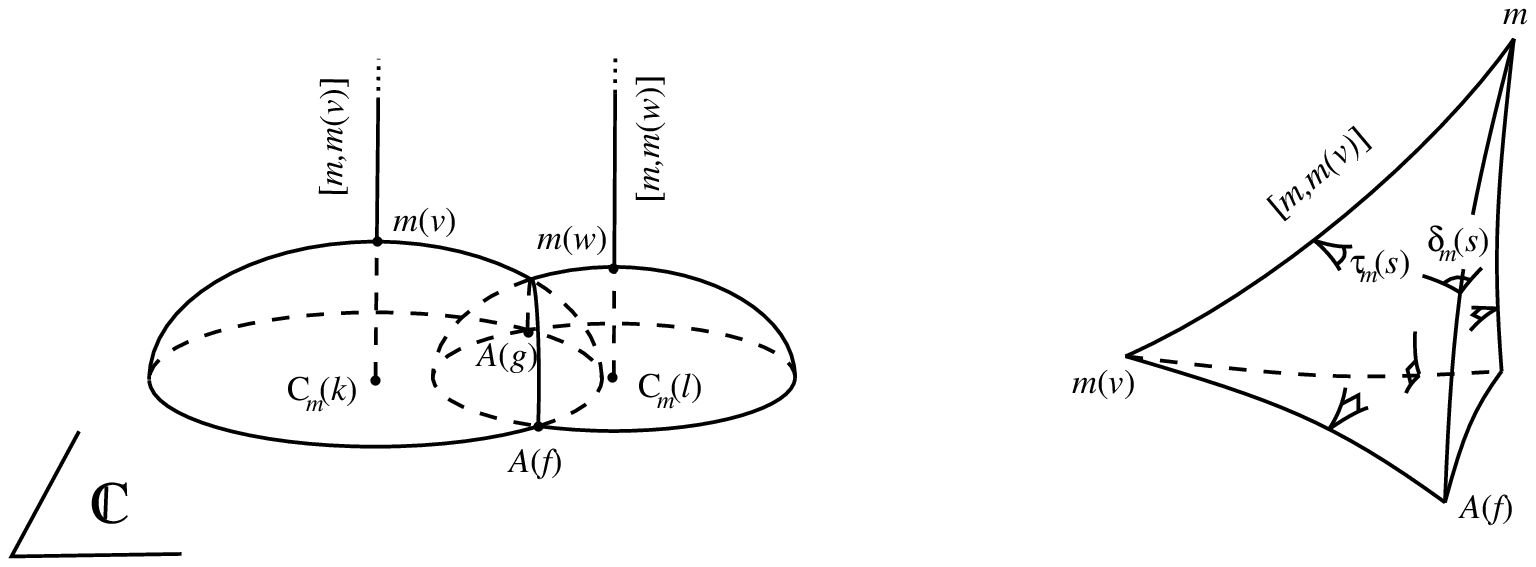}
{\vspace{-2ex}\hspace*{6em}Figure~\ref{orthoschem}{\rm{a}}\hspace*{10em}Figure~ \ref{orthoschem}{\rm{b}}}{0}{5}

The hyperplane ${\cal H}$ passing 
through $m$, $k$ and $l$  
divides 
$\Fig m {\wt s}$ in two congruent simplices. Let $\Ort m {\wt s}$ be the one containing 
$\ha$. 
Since ${\cal H}$ is perpendicular to $\H k$ and $\H l$, the geodesic line 
$\dual{\proj {\wt s}}=\H k\cap\H l$ intersects ${\cal H}$ perpendicularly.  
%The 
%geodesic line which passes throught $m$ and $k$ intersects 
%$\H k$ perpendicularly.  Moreover, the geodesic $\dual{\proj {\wt s}}$ is 
%perpendicular to ${\cal H}$. 
Hence,
all but at most three dihedral 
angles of $\Ort m {\wt s}$ are right.
Such simplices are called \new{orthoschemes}. 
Figure~\ref{orthoschem}{b} shows a schematic view of $\Ort m {\wt s}$. 
\newline
In addition, the sum of the dihedral angles  at the vertex 
$\ha$ is $\pi$.  
In fact,  
$\Lim {_{\ha}} {\Ort m {\wt s}}$ is  a geodesic triangle $d$ 
in the Euclidean plane $(\CS\setminus\ha,\ha)$ and  the angles of $d$ coincide with the 
dihedral angles of $\Ort m {\wt s}$  at the vertex $\ha$ (see~\ref{dihedral}). 
%Thus, the volume of $\Ort m {\wt s}$ depends only
% on two  dihedral angles at the `vertex' $m$. 
If 
$\tau_m(s)$ (respectively, $\delta_m(s)$) denotes the dihedral angles  
at the edge carried by the geodesic line passing through $m$ and $k$  (respectively, $m$ and $\dual f$), then the following formula holds 
(see [Kh]):
$$
\vol \Ort m {\wt s}=\Fau{\tau_m(s),\delta_m(s)},\quad\mbox{where}$$
\begin{equation}\label{Vol_formel}
\Fau{x,y}:=
\frac 14\lob\left(x+\frac\pi 2-y\right)
+\frac 14\lob\left(-x+\frac \pi 2-y \right)+
\frac 12\lob\left(y\right).
\end{equation}

We first determine the angle $\tau_m(s)$. 
In \ref{dihedral} we showed that the dihedral angles of $\Ort m{\wt s}$ at the 
`vertex' $m$ coincide with the angles of the geodesic triangle $\Lim m{\Ort m {\wt s}}$
 in $(\reg m, m)$. 
Every leg of this triangle is contained in the boundary of a 
hyperplane carrying a facet of $\Ort m{\wt s}$. Hence, 
the angle $\tau_m(s)$ coincides with an   
angle enclosed by the geodesic line in $(\reg m,m)$ passing through 
$\Mi m k$, $\Mi m l$ and the geodesic line passing through $\Mi m k$, $\ha$, i.e. $\tau_m(s)=\psi_m(s)$ or $\tau_m(s)=\pi-\psi_m(s)$. 
Since $2\cdot\tau_m(s)$ is a dihedral angle
 of the convex set $\Fig m {\wt s}$, the angle
$\tau_m(s)$ cannot 
be bigger than $\pi/2$. 
Thus,
\begin{equation}\label{elegant}
\tau_m(s):=\left\{
\begin{array}{lll}
\psi_m(s)&\mbox{if}&\psi_m(s)\leq\frac\pi 2,\\
\pi-\psi_m(s)&\mbox{if}&\psi_m(s)\geq\frac\pi 2.
\end{array}
\right.
\end{equation}

Our next step will be to express the index 
$\Index{m(v)}{\dual {\proj {\wt s}}}{\fac v}$ in terms of the triangle  
$\Delta_m(\proj s)$. 
The geodesic line $\dual{\proj{\wt s}}$ divides $\H k$ in two half-planes
 (see Figure~\ref{orthoschem}{a}). 
We have $\Index{m(v)}{\dual {\proj {\wt s}}}{\fac v}=1$ 
(respectively, $-1$) if and only if $m(v)$ is contained in the 
open half-plane carrying $\fac v\setminus \dual{\proj{\wt s}}$ 
(respectively, the half-plane containing no point of $\fac v$). 
In order to relate $\Index{m(v)}{\dual {\proj {\wt s}}}{\fac v}$ with 
$\psi_m(s)$ we project $m(v)$, $\dual {\proj {\wt s}}$ and $\H k$ to 
$\partial\MR$. We have  
$$\Lim m {m(v)}\in\Lim m {\H k}\supset\Lim m {\fac v}
%=\Lim m {\Cone m {\fac v}}.$$
\supset\Lim m{\dual {\proj {\wt s}}}.$$
The set $\Lim m {\H k}$ is again a conformal disk contained in $\reg m$ and 
 bounded by $\partial\reg k$. 
We denote the disk metric with regular domain 
$\Lim m {\H k}$ by $k^*$.  
Since $\H k=\H{k^*}$, the geodesic line  carrying $[m,m(v)]=[m,k]$   
 passes through $\Mi m k$ and $\Mi m {k^*}=\Lim m {m(v)}$. 
%Note that $\Mi m k=\Mi m {k^*}$ if and only if $\Index{m}{\fac v}{\CH\cK}=1$. 
%In particular, if $m\in\partial\MR\cup\MS$, then $\Mi m n=\Mi m {n^*}$.
 \newline
The geodesic line $\Lim m {\dual{\proj{\wt s}}}$ in $(\reg m, m)$ divides 
the disk $\reg {k^*}$ in two open half-disks. Let $d$ be the one containing 
no point of $\Lim m {\fac v}$ (Figure~\ref{index_mv}). Then  
$$\Index{m(v)}{\dual {\proj {\wt s}}}{\fac v}=
\left\{\begin{array}{rl}
            -1 &\mbox{if}\quad \Mi{m}{k^*}\in d,\\
            0 &\mbox{if}\quad\Mi{m}{k^*}\in \Lim m {\dual{\proj{\wt s}}},\\
            +1 &\mbox{else.}
\end{array}\right.
$$
In \ref{dihedral} we showed that any geodesic line  in $(\reg m,m)$ containing  
the point $\Mi m k$ is of the form $\Lim m {M}$, where ${M}$ is a hyperplane 
in $\MR$ passing through $m$ and $k$. Since these hyperplanes passes also 
through $\Mi m {k^*}$, every geodesic through $\Mi m k$ contains the point 
$\Mi m {k^*}$. In particular we conclude that $\Mi m k=\Mi m {k^*}$ if 
$m\in\partial\MR\cup\MD$. It is not difficult to verify that  
$\Mi m k=\Mi m {k^*}$ if and only if $\Index{m}{\fac v}{\CH\cK}=1$. 
\newline
For pairwise distinct points $A,B,C\in \reg m$ let $[A,B]_m$ denote the geodesic segment in 
$(\reg m,m)$ joining $A$, $B$ and $\measuredangle_m(A,B,C)$ the angle  of the triangle $A,B,C$ at the vertex $B$. We have 
$\measuredangle_m(A,\Mi m k,C)=\measuredangle_m(A,\Mi m {k^*},C)$
 (see Figure~\ref{bian_deg}a if $B:=\Mi m k\neq\Mi m {k^*}=:B'$). 
Hence, 
 $2\psi_m(s)$ is just the angle obtained by 
turning $[\Mi{m}{k^*},\ha]_m$ into $[\Mi{m}{k^*},\hap]_m$ without passing through a 
vertex of $\Lim m {\fac v}$ (Figure~\ref{index_mv}) and 
\begin{equation}\label{Fuetzli}
\Index{m(v)}{\dual {\proj {\wt s}}}{\fac v}=
\left\{\begin{array}{rl}
            1 &\mbox{if}\quad\psi_m(s)<\frac\pi 2,\\
            0 &\mbox{if}\quad\psi_m(s)=\frac\pi 2,\\
            -1 &\mbox{if}\quad\psi_m(s)>\frac\pi 2.
\end{array}\right.
\end{equation}

\Figur{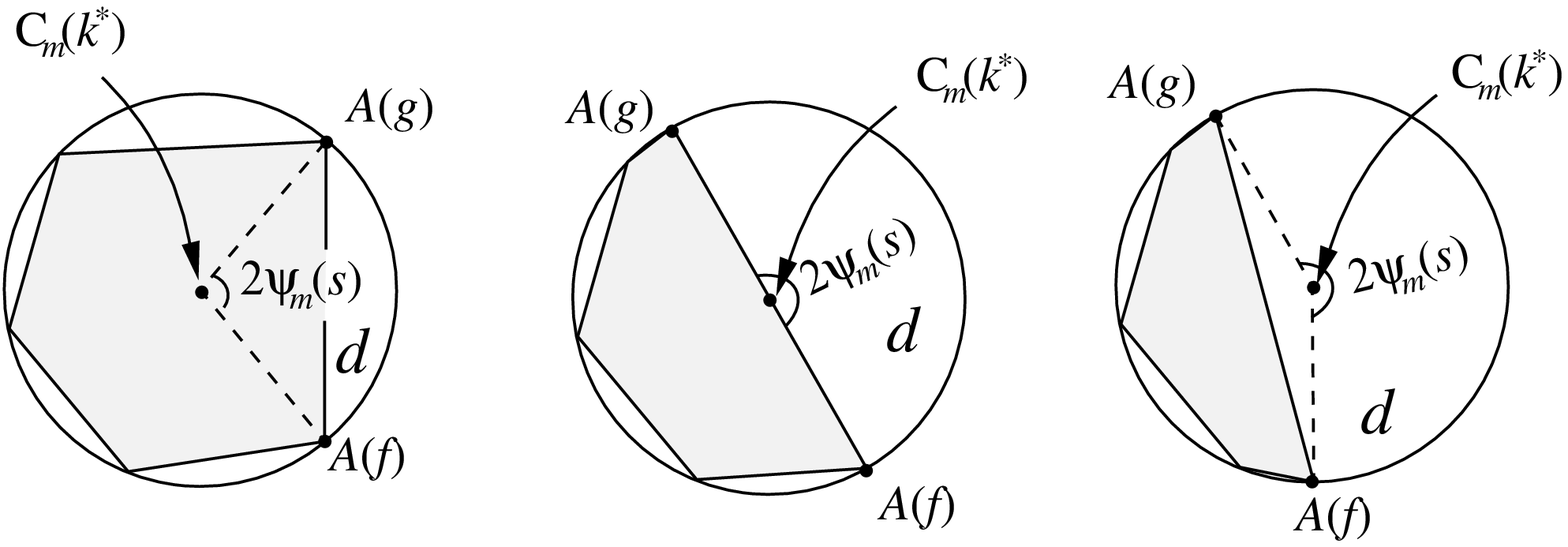}
{\vspace*{-3em}
$$\Index{m(v)}{\dual {\proj {\wt s}}}{\fac v}=1\hspace{2em}
\Index{m(v)}{\dual {\proj {\wt s}}}{\fac v}=0\hspace{2em}
\Index{m(v)}{\dual {\proj {\wt s}}}{\fac v}=-1$$
\vspace*{-1.5ex}
\newline
Figure~\ref{index_mv}}{-1}{5}

Finally we have to determine the angle $\delta_m(s)$.  
Let 
$\Delta$ be the triangle in the congruence class $\Delta_m(\proj s)$ with 
vertices $\Mi m k$, $\Mi m l$, $\ha$, and 
 $\Dest$ the  triangle contained in $\reg{k^*}\cup\reg{l^*}$ with
 vertices $\Mi m {k^*}$, $\Mi m {l^*}$, $\ha$.  
To simplify the notation 
we denote the angles $\psi(s),\psi(-s),\pi-\theta(\proj s)$ of $\Delta$ by 
$\alpha,\beta,\gamma$ and the angles of $\Dest$ at the vertices  
 $\Mi m {k^*},\Mi m {l^*},\ha$ by $\alpha^*,\beta^*,\gamma^*$. 
 If $m\in\MR$, then 
Figure~\ref{delts_stern3_oh}$\ $  illustrates the relations between the angles 
$\alpha,\beta,\gamma$ and $\alpha^*,\beta^*,\gamma^*$. 
It shows the triangles 
$\Delta$ and $\Delta^*$ in the Poincar\'e model 
 $\MR=\{x\in{\bbb R}^3\mid |x|<1\}$ with $m=0\in{\bbb R}^3$, 
i.e. $(\reg m,m)$ is the standard metric sphere in ${\bbb R}^3$.
\subsubsection*{}
%\epsfxsize=0.8\figurwidth
%\Figur{dreieck*}
\Figur{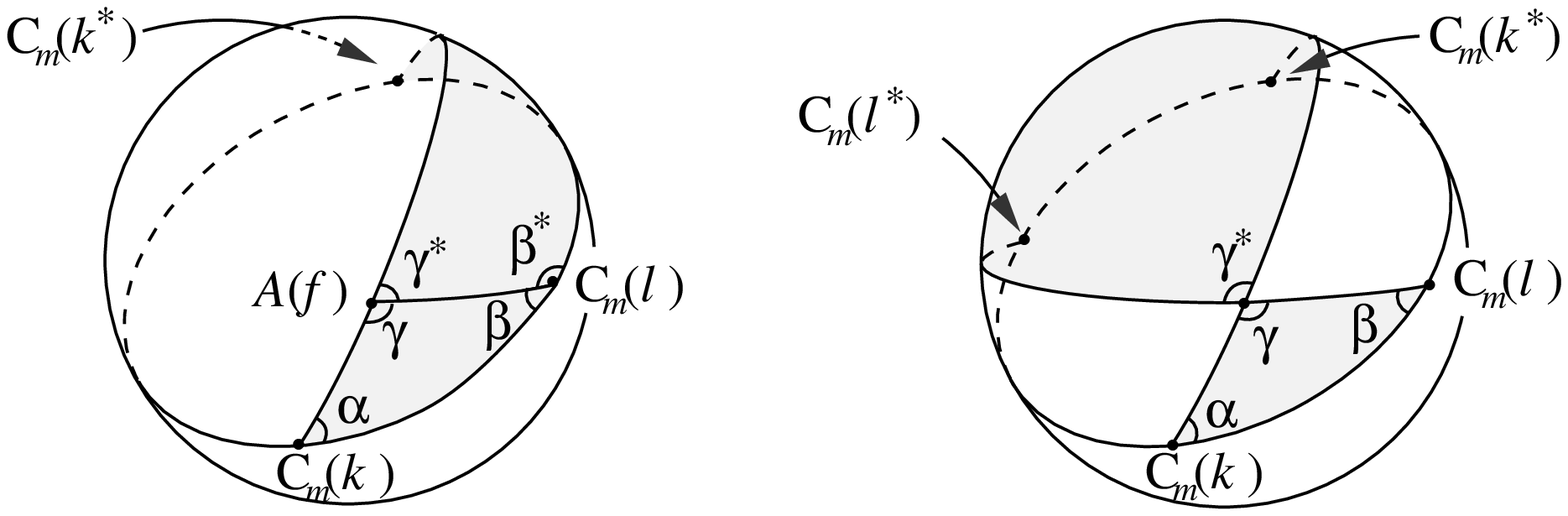}
{\vspace*{-2em}$$\Mi m{l^*}=\Mi m l,\ \Mi m{k^*}\neq\Mi m{k}\hspace{2em}
\Mi m{l^*}\neq\Mi m l,\ \Mi m{k^*}\neq\Mi m{k}$$
\vspace*{-1em}
\newline
Figure~\ref{delts_stern3_oh}{\rm a}\hspace*{10em}Figure~\ref{delts_stern3_oh}{\rm b}}{0}{4}

\noindent
If $m\in\partial\MR\cup\MD$, then  $\Dest=\Delta$.  
Since $\Lim m{\Fig m{\wt s}}$ 
is the convex hull  of the set $\{\ha,\hap,\Mi m{k^*}\}$ in $(\reg m, m)$, the angle 
$\delta_m(s)$ is  bounded by the altitude of $\Dest$ and the leg 
passing through  $\Mi m {k^*}$, $\dual f$ (Figure~\ref{delta}). 
%Note that $\delta_m(s)<\pi/2$ since $\Index{m}{\fac v}{\CH\cK}\neq 0$, i.e. 
%$m\not\in\H k$.
%\epsfxsize=0.8\figurwidth

\Figur{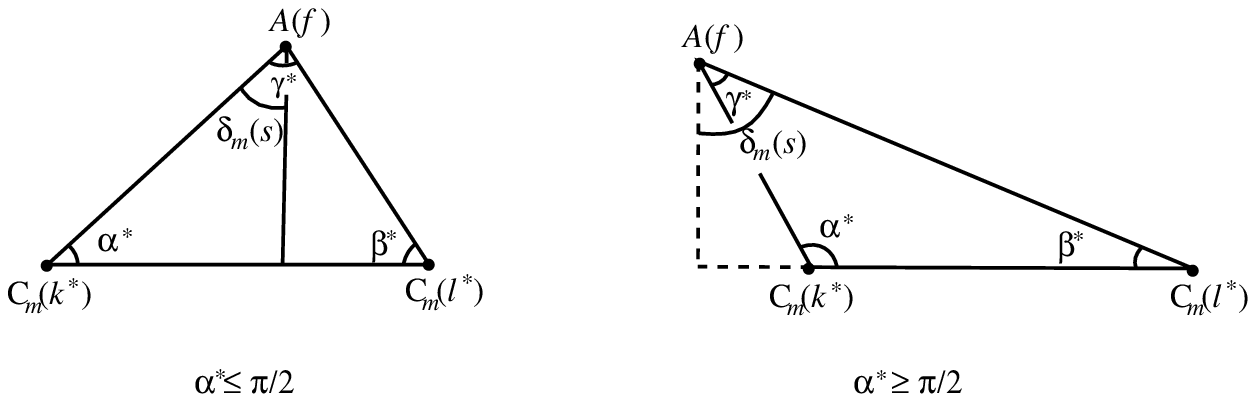}{
%\vspace*{-2.5em}
%$$\alpha^*\leq\frac\pi 2\hspace*{12em}\alpha^*\geq\frac\pi 2$$
%\vspace*{-1em}
%\newline
Figure~\ref{delta}}{-4}{4}

\noindent
A short computation using the trigonometric relations in the triangle $\Dest$ 
yields the following formula:
\begin{equation}
\delta_m(s)\label{auch_das_noch}=
\left\{\begin{array}{rl}
            \wabst &\mbox{if}\quad \alpha^*\leq\frac\pi 2,\\
            -\wabst &\mbox{if}\quad \alpha^*\geq\frac\pi 2,
\end{array}\right.\\
\end{equation}
where 
\begin{eqnarray*}
\omega_z(x,y):=\arctan\left(\frac{\cos x\sin z}{\cos y+\cos x\cos z}\right)
\in (-\frac\pi 2,\frac\pi 2).
\end{eqnarray*}
From the definition of $\Delta$, $\Dest$ and the fact that 
$$\omega_{\pi-\gamma}(\pi-\alpha,\beta)=
\omega_{\pi-\gamma}(\alpha,\pi-\beta)=-\omega_\gamma(\alpha,\beta)$$
 we immediately 
deduce the following relations (Figure~\ref{delts_stern3_oh}):
\medskip
\begin{center}
{\small
\begin{tabular}{|c|c||c|c|c|c||c|}
\hline
$\Index{m}{\fac v}{\CH\cK}$&$\Index{m}{\fac w}{\CH\cK}$&
$\alpha^*$&$\beta^*$&$\gamma^*$&$\wabst$&Remark\\
\hline\hline
$+1$&$+1$&$\alpha$&$\beta$&$\gamma$&$+\wab$&$\Delta=\Delta^*$\\
$-1$&$+1$&$\alpha$&$\pi-\beta$&$\pi-\gamma$&$-\wab$&Fig. \ref{delts_stern3_oh}a\\
$-1$&$-1$&$\pi-\alpha$&$\pi-\beta$&$\gamma$&$+\wab$&Fig. \ref{delts_stern3_oh}b\\
$+1$&$-1$&$\pi-\alpha$&$\beta$&$\pi-\gamma$&$-\wab$&\\
\hline
\end{tabular}
}
\end{center}
\medskip
In particular we conclude that 
$$\wabst=\Index{m}{\fac v}{\CH\cK}\cdot\Index{m}{\fac w}{\CH\cK}\cdot\wab$$
and with (\ref{Fuetzli}) we get 
\begin{eqnarray*}
\alpha^*<\frac\pi 2&\Longleftrightarrow
&\Index{m}{\fac w}{\CH\cK}\cdot\Index{m(v)}{\dual {\proj {\wt s}}}{\fac v}=+1\\
\alpha^*>\frac\pi 2&\Longleftrightarrow
&\Index{m}{\fac w}{\CH\cK}\cdot\Index{m(v)}{\dual {\proj {\wt s}}}{\fac v}=-1.\\
\end{eqnarray*}
Combining this with (\ref{auch_das_noch}) yields  
$$\delta_m(s)=\varepsilon_m(s)\cdot\wab.$$
Since $\Fau{x,\pm y}=\Fau{\pi-x,\pm y}=\pm\Fau{x,y}$, we finally get 
\begin{eqnarray*}
\Fau{\alpha,\wab}&=&\Fau{\tau_m(s),\wab}=\varepsilon_m(s)\Fau{\tau_m(s),\delta_m(s)}
=\varepsilon_m(s)\vol\Ort m {\wt s}\\
&=&
\frac{\varepsilon_m(s)}2\vol\Fig m s.
\end{eqnarray*}
\vspace*{-0.3em}
The technical Lemma~\ref{lem_tec} completes the proof.
\medskip\newline\indent
If $\Index{m}{\fac v}{\CH\cK}\cdot\Index{m}{\fac w}{\CH\cK}$ is zero, then  
$m\in \H{k}\cup \H{l}\subset\MR$. We proved that for every $m'\in\MR$ with $m'\not\in  
\H{k}\cup\H{l}$ the volume formula 
$$
\varepsilon_{m'}(s)\cdot\vol\Fig {m'} s + 
\varepsilon_{m'}(-s)\cdot\vol\Fig {m'} {-s}=
\Ino{\theta (\proj s)}(\wh\psi_{m'}(s))
 -\Imi{\theta(\proj s)}(\ex_{m'}(\proj s))
$$
holds. Since both sides of this formula are continuous in $\MR$, 
it still holds if $m=m'$.
\end{beweis}
%%%%%%%%%%%%%%%%%%%%%%%%%%%%%%%%%%%%%%%%%
%

\begin{lemma}\label{lem_tec}
For $\gamma\in{\bbb R}$ we define %the functions
$\omega_\gamma, {\cal V}:{\bbb R}^2\longrightarrow {\bbb R}$ by 
\begin{eqnarray*}
\ \omega_\gamma(x,y)&:=&
\arctan\left(\frac{\cos x\sin \gamma}{\cos y+\cos x\cos \gamma}\right)
\\
\Fau{x,y}&:=&
\frac 14\lob\left(x+\frac\pi 2-y\right)
+\frac 14\lob\left(-x+\frac \pi 2-y \right)+
\frac 12\lob\left(y\right).
\end{eqnarray*}
If $\alpha$, $\beta\in{\bbb R}$, then
\begin{eqnarray*}
\lefteqn{
2\,\Fau{\alpha,\wab}+2\,\Fau{\beta,\wba}=
}\nonumber\\
& &
\Imi{\gamma}\left(\frac{\alpha-\beta-\gamma+\pi}2\right)
-\Ino\gamma\left(\frac{\alpha+\beta+\gamma-\pi}2\right).
\end{eqnarray*}
\end{lemma}

\begin{beweis}
It is understood that the function $\arctan$ is defined on ${\bbb R}\cup\{+\infty,-\infty\}$,
 i.e. we define $\arctan(+\infty)=\pi/2$ and $\arctan(-\infty)=-\pi/2$. 
Let 
$$\ex:=(\alpha+\beta+\gamma-\pi)/2,\quad 
\wh\alpha:=\alpha-\ex,\quad\wh\beta:=\beta-\ex,\quad\wh\gamma:=\gamma-\ex.$$ 
Proposition~\ref{prop2} in \ref{functional} yields  
$$\Imi{\gamma}\left(\wh\alpha\right)
-\Ino\gamma\left(\ex\right)=\lob(\wh\alpha)+\lob(\wh\beta)-\lob(\ex)-
\lob(\wh\gamma)+\lob(\gamma).$$
Hence, we have to show that 
\begin{equation}\label{eqn_AAA}
2\,\Fau{\alpha,\wab}+2\,\Fau{\beta,\wba}=
\lob(\wh\alpha)+\lob(\wh\beta)-\lob(\ex)-
\lob(\wh\gamma)+\lob(\gamma).
\end{equation}
Note that both sides of (\ref{eqn_AAA})  
are continuous in $\alpha$, $\beta$, $\gamma$
and that
\begin{equation}\label{eqn_AA1}
\wab+\wba=\gamma\mod\pi. 
\end{equation}
We will prove (\ref{eqn_AAA}) if $\wab\equiv 0\mod \pi$ and  we will 
show that the partial derivatives of both sides of (\ref{eqn_AAA}) 
coincide almost everywhere.  
\newline
We have $\wab\equiv 0\mod \pi$ if and only if 
$\alpha\equiv \pi/2\mod\pi$ or $\gamma\equiv 0\mod\pi$. 
If $\gamma\equiv 0\mod\pi$ both sides of (\ref{eqn_AAA}) are zero. 
If $\alpha\equiv \pi/2\mod\pi$ we get $\Fau{\alpha,0}=0$ and 
Proposition~\ref{prop2} yields 
$$\lob(\alp)-\lob(\ex)
=
\lob\left(\frac\pi 2\pm\frac\pi 4-\frac\beta 2-\frac\gamma 2\right)
+\lob\left(\pm\frac\pi 4-\frac\beta 2-\frac\gamma 2\right)
=\frac12\lob(-\beta+\frac\pi 2-\gamma )
$$
$$
\!\!\!\!\lob(\bet)-\lob(\gam)
=
\lob\left(\frac\beta 2-\frac\gamma 2\pm\frac\pi 4\right)+
\lob\left(\frac\pi 2+\frac\beta 2-\frac\gamma 2\pm\frac\pi 4\right)
=\frac12\lob(\beta+\frac\pi 2-\gamma).
$$
Thus, the right side of (\ref{eqn_AAA}) is just $2\Fau{\beta,\gamma}$. 
\newline
It is not difficult to verify that 
both sides of (\ref{eqn_AAA}) are differentiable if 
$$\wh\alpha,\wh\beta,\wh\gamma,\ex\not\equiv 0\mod \pi 
\quad\mbox{and}\quad \wab,\wba\not\equiv 0\mod\pi/2.$$ 
In this case we will calculate the derivatives in direction 
$\alpha,\beta$ and $\gamma$. We first state the following 
two identities which can be verified with the addition formulas 
in trigonometry:
\begin{eqnarray*}
\left|\frac
{\sin(\alpha+\frac\pi 2-\wab)}{\sin(-\alpha+\frac\pi 2-\wab)}
\right|&=&
\left|\frac{\sin{\wh\alpha} \sin{\wh\gamma} }
{\sin{\ex}\sin{\wh\beta} }\right|
,
\\
\left|\frac
{\sin(\alpha+\frac\pi 2-\wab)\sin(-\alpha+\frac\pi 2-\wab)}
{\sin^2\wab}
\right|&=&
\left|\frac{4\sin{\wh\alpha}  \sin{\wh\gamma} 
\sin{\ex}  \sin{\wh\beta} }{\sin^2\gamma}\right|.
\end{eqnarray*} 
Hence, we get 
\begin{eqnarray*}
4\frac\partial{\partial\alpha}\Fau{\alpha,\wab}&=&
-\log\left|\frac{\sin{\wh\alpha}  \sin{\wh\gamma} }
{\sin{\ex}  \sin{\wh\beta} }\right|
+\frac{\partial\wab}{\partial\alpha}
\log
\left|\frac{4\sin{\wh\alpha}  \sin{\wh\gamma} 
\sin{\ex} \sin{\wh\beta} }{\sin^2\gamma}\right|\\
4\frac\partial{\partial\alpha}\Fau{\beta,\wba}&=&
\frac{\partial\wba}{\partial\alpha}
\log
\left|\frac{4\sin{\wh\alpha}  \sin{\wh\gamma} 
\sin{\ex}  \sin{\wh\beta} }{\sin^2\gamma}\right|,\\
4\frac\partial{\partial\gamma}\Fau{\alpha,\wab}&=&
\frac{\partial\wab}{\partial\gamma}
\log
\left|\frac{4\sin{\wh\alpha} \sin{\wh\gamma} 
\sin{\ex}  \sin{\wh\beta} }{\sin^2\gamma}\right|.
\end{eqnarray*}
Using (\ref{eqn_AA1}), the derivatives of the left side of 
(\ref{eqn_AAA}) in direction 
$\alpha$, $\beta$ and $\gamma$ reduces to 
$$
-\frac12\log\left|\frac{\sin{\wh\alpha}  \sin{\wh\gamma} }
{\sin{\ex}  \sin{\wh\beta} }\right|,
\qquad
-\frac12\log\left|\frac{\sin{\wh\beta}  \sin{\wh\gamma} }
{\sin{\ex}  \sin{\wh\alpha} }\right|,
\qquad
\frac12\log
\left|\frac{4\sin{\wh\alpha}  \sin{\wh\gamma} 
\sin{\ex}  \sin{\wh\beta} }{\sin^2\gamma}\right|.
$$
Therefore, they  coincide with the derivatives of the right 
side of 
(\ref{eqn_AAA}). 
%This completes the proof
\end{beweis}
\bigskip

%\section*{References}

\bigskip
E-mail address:\quad
{\tt{walter.braegger[at]gmail.com}}

\begin{thebibliography}{Kreisp}
\addcontentsline{toc}{section}{References}
\bibitem[A1]{An}
E.M. Andreev, On convex polyhedra in Lobacevskii spaces.
{\em Math. USSR Sb.} {\bf 10} (1970), 413-440.
\bibitem[A2]{An2}
E.M. Andreev, On convex polyhedra of finite volume
 in Lobacevskii spaces.
{\em Math.USSR Sb.} {\bf 12} (1970), 255-259.
\bibitem[Be]{Be}
Claude Berge, ``Graphs". %et hypergraphes. 
North-Holland Mathematical Library%; Vol. 6, part 1
,
Amsterdam, 1991.
\bibitem[Br]{Br}
Walter Br\"agger, Kreispackungen und Triangulierungen. 
{\em L'Enseignement Math\'ematique} {\bf 38} (1992), 201-217.
\bibitem[CV]{Cv}
Yves Colin de Verdi$\grave{\rm e}$re, Un principe variationnel pour les 
empilements de cercles.
{\em Invent. Math.} {\bf 104} (1991), 655-669.
\bibitem[Ho]{Ho}
C.D. Hodgson, I.Rivin, W.D. Smith, A characterization of convex 
hyperbolic polyhedra and of convex polyhedra inscribed in the sphere.
 {\em Bull. Am. Math. Soc.} {\bf 27} (1992), 246-251.
\bibitem[Kh]{Kh}
R. Kellerhals, On the volume of hyperbolic polyhedra. 
{\em Math. Ann.} {\bf 285} (1989), 541-569.
\bibitem[Mi]{Mi}
J. W. Milnor, 
Hyperbolic geometry: the first 150 years.
{\em Bull. Am. Math. Soc.} {\bf 6} (1982), 9-24.
\bibitem[R1]{Ria}
Igor Rivin, 
A characterization of ideal polyhedra in hyperbolic 3-space. 
 preprint, 1992.
\bibitem[R2]{Rib}
Igor Rivin, 
Euclidean structures on simplicial surfaces and hyperbolic volume. 
{\em Ann. of Math.} {\bf 139} (1994), 553-580.
\bibitem[Tu]{Tu}
 William P. Thurston, 
``The Geometry and Topology of Three-Manifolds".
 Lecture Notes, Dep't. of Math., Princeton University, Princeton, 1978. 
% chap.13
\end{thebibliography}
\end{document}